\documentclass[12pt]{article}

\usepackage[T1]{fontenc}

\usepackage{amsmath}
\usepackage{amssymb}
\usepackage{url, enumerate}
\usepackage{hyperref} 
\usepackage{verbatim}
\usepackage[utf8]{inputenc} 
\usepackage[T1]{fontenc}    
\fontencoding{T1}\selectfont

\newcommand{\eop}{\bigstar}  

\newcommand{\cf}{{\rm cf}}

\newcommand{\D}{{\mathcal D}}
\newcommand{\B}{{\mathcal B}}
\newcommand{\OO}{{\mathcal O}}

\newcommand{\la}{\langle}
\newcommand{\ra}{\rangle}

\newcommand{\bigbox}{\square}

\newenvironment{proof}{\noindent{\bf Proof.}}{\par\bigskip}

\newenvironment{proof-}{\noindent{\bf Proof}}{\par\bigskip}

\newtheorem{THEOREM}{Theorem}[section]

\newtheorem{Conclusion}[THEOREM]{Conclusion}

\newtheorem{Hypothesis}[THEOREM]{Hypothesis}

\newtheorem{LEMMA}[THEOREM]{Lemma}

\newtheorem{Main Theorem}[THEOREM]{Main Theorem}
\newenvironment{main Theorem}{\begin{Main Theorem}} 
{\end{Main Theorem}}

\newtheorem{Theorem}[THEOREM]{Theorem}
\newenvironment{theorem}{\begin{Theorem}}{\end{Theorem}}
\newenvironment{thm}{\begin{Theorem}}{\end{Theorem}}

\newtheorem{Definition}[THEOREM]{Definition}
\newenvironment{definition}{\begin{Definition}}{\end{Definition}}
\newenvironment{defn}{\begin{Definition}}{\end{Definition}}

\newtheorem{Conventions}[THEOREM]{Conventions}

\newtheorem{Main Definition}[THEOREM]{Main Definition}
\newenvironment{main definition}{\begin{Main Definition}}
{\end{Main Definition}}

\newtheorem{Lemma}[THEOREM]{Lemma}
\newenvironment{lemma}{\begin{Lemma}}{\end{Lemma}}

\newtheorem{Notation}[THEOREM]{Notation}

\newtheorem{Convention}[THEOREM]{Convention}

\newtheorem{Note}[THEOREM]{Note}

\newtheorem{Observation}[THEOREM]{Observation}
\newenvironment{observation}{\begin{Observation}}
{\end{Observation}}

\newtheorem{Remark}[THEOREM]{Remark}

\newtheorem{Question}[THEOREM]{Question}
\newenvironment{question}{\begin{Question}}{\end{Question}}

\newtheorem{Main Fact}[THEOREM]{Main Fact}
\newenvironment{main Fact}{\begin{Main Fact}}{\end{Main Fact}}

\newtheorem{Fact}[THEOREM]{Fact}

\newtheorem{Subfact}[THEOREM]{Subfact}

\newtheorem{Claim}[THEOREM]{Claim}
\newenvironment{claim}{\begin{Claim}}{\end{Claim}}

\newtheorem{Main Claim}[THEOREM]{Main Claim}
\newenvironment{main claim}{\begin{Main Claim}}{\end{Main Claim}}

\newtheorem{Crucial Claim}[THEOREM]{Crucial Claim}
\newenvironment{crucial claim}{\begin{Crucial Claim}}{\end{Crucial Claim}}

\newtheorem{Subclaim}[THEOREM]{Subclaim}

\newtheorem{Sublemma}[THEOREM]{Sublemma}

\newtheorem{Corollary}[THEOREM]{Corollary}
\newenvironment{corollary}{\begin{Corollary}}{\end{Corollary}}
\newenvironment{cor}{\begin{Corollary}}{\end{Corollary}}

\newtheorem{Example}[THEOREM]{Example}
\newenvironment{example}{\begin{Example}}{\end{Example}}

\newtheorem{Problem}[THEOREM]{Problem}

\newtheorem{Proposition}[THEOREM]{Proposition}

\newtheorem{Conjecture}[THEOREM]{Conjecture}
\newenvironment{conjecture}{\begin{Conjecture}}{\end{Conjecture}}

\newtheorem{Discussion}[THEOREM]{Discussion}

\newenvironment{Proof of the Subfact}
{\noindent{\bf Proof of the Subfact.}}{\par\bigskip}

\newenvironment{Proof of the Theorem}
{\noindent{\bf Proof of the Theorem.}}{\par\bigskip}

\newenvironment{Proof of the Proposition}
{\noindent{\bf Proof of the Proposition.}}{\par\bigskip}

\newenvironment{Proof of the Conclusion}
{\noindent{\bf Proof of the Conclusion.}}{\par\bigskip}

\newenvironment{Proof of the Observation}
{\noindent{\bf Proof of the Observation.}}{\par\bigskip}

\newenvironment{Proof of the Fact}
{\noindent{\bf Proof of the Fact.}}{\par\bigskip}

\newenvironment{Proof of the Lemma}
{\noindent{\bf Proof of the Lemma.}}{\par\bigskip}

\newenvironment{Proof of the Claim}
{\noindent{\bf Proof of the Claim.}}{\par\bigskip}

\newenvironment{Proof of the Corollary}
{\noindent{\bf Proof of the Corollary.}}{\par\bigskip}

\newenvironment{Proof of the Subclaim}
{\noindent{\bf Proof of the Subclaim.}}{\par\medskip}

\newenvironment{Proof of the Main Claim}
{\noindent{\bf Proof of the Main Claim.}}{\par\bigskip}

\newenvironment{Proof of the Crucial Claim}
{\noindent{\bf Proof of the Crucial Claim.}}{\par\bigskip}

\newcommand{\pr}{{\rm pr}}

\newcommand{\elementary}{\prec}

\newcommand{\rest}{\upharpoonright}  

\newcommand{\supt}{\mathop{\rm supt}}

\newcommand{\deq}{\buildrel{\rm def}\over =}

\newcommand{\BB}{{\mathcal B}}
\newcommand{\CC}{{\mathcal C}}

\newcommand{\FF}{{\mathcal F}}
\newcommand{\F}{{\mathcal F}}
\newcommand{\GG}{{\mathcal G}}
\newcommand{\HH}{{\mathcal H}}

\newcommand{\I}{{\mathcal I}}

\newcommand{\LL}{{\mathcal L}}

\newcommand{\RR}{{\mathcal R}}

\newcommand{\UU}{{\mathcal U}}
\newcommand{\U}{{\mathcal U}}
\newcommand{\VV}{{\mathcal V}}
\newcommand{\W}{{\mathcal W}}
\newcommand{\WW}{{\mathcal W}}

\def\empty{\emptyset}

\newcount\skewfactor
\def\mathunderaccent#1#2 {\let\theaccent#1\skewfactor#2
\mathpalette\putaccentunder}
\def\putaccentunder#1#2{\oalign{$#1#2$\crcr\hidewidth
\vbox to.2ex{\hbox{$#1\skew\skewfactor\theaccent{}$}\vss}\hidewidth}}

\newcommand{\dom}{\mbox{\rm dom}}

\begin{document}

\title{On middle box products and paracompact cardinals}

\author{{David Buhagiar}\footnote{Department of
Mathematics, Faculty of Science, University of Malta, Msida, MSD2080, Malta, \scriptsize{david.buhagiar@um.edu.mt} \url{https://www.um.edu.mt/profile/davidbuhagiar}}
\,and {Mirna D\v zamonja}\footnote{IRIF,
CNRS et Université de Paris Cité,
8 Place Aurélie Nemours
75205 Paris Cedex 13, France, \scriptsize{mdzamonja@irif.fr}, \url{https://www.logiqueconsult.eu}}}

\maketitle

\begin{abstract} The paper gives several sufficient conditions on the paracompactness of box products with an arbitrary number of factors and boxes of arbitrary size. The former include results on generalised metrisability and Sikorski spaces. Of particular interest are products of the type ${}{<\kappa} \square \,2^\lambda$, where we prove that for a regular uncountable cardinal $\kappa$, if ${}^{<\kappa} \square \, 2^\lambda$ is paracompact for every $\lambda\ge\kappa$, then $\kappa$ is at least inaccessible.
The case of the products of the type ${}^{<\kappa} \square\, X^\lambda$ for $\kappa$ singular has not been studied much in the literature and we offer various results.
The question if ${}^{<\kappa}\square\, 2^\lambda$ can be paracompact for all $\lambda$ when $\kappa$ is singular has been partially answered but
remains open in general.

\end{abstract}

\thanks{Mirna D\v zamonja's research was supported through the ERC H2020 project FINTOINF-MSCA-IF-2020-101023298. She thanks IHPST at the Université de Panthéon-Sorbonne, Paris, where she is an Associate Member, for their permanent support, and the University of Malta for their support as a Visitor in October 2020, as well as the University of East Anglia where she is a Visiting Professor.}

{\em Keywords}: box products, paracompact cardinals, singular cardinals

{\em MSC2020 Classification}: 03E10, 03E55, 54A10, 54D20


\section{Introduction}\label{sec:intro} 

\subsection{Motivation}
This paper concerns an old subject of the preservation of variants of compactness by variants of products. Ken Kunen gave a seminal contribution to this area through his paper \cite{MR514975},
where he proved 

\begin{theorem}[Kunen's Theorem]\label{Kunenth} [CH] A box product of countably many compact spaces is paracompact iff its Lindel\"of degree is $\le \mathfrak c$.
\end{theorem}

Kunen's Theorem is one of the highlights not only of Ken's work but also of that of the whole Madison school of
set-theoretic topology. This group included permanent members Ken, Arnie Miller and Mary Ellen Rudin, members of neighbouring univerities such as Paul Bankston in Milwaukee, distinguished visitors such as Adam Ostazewski, Judy Roitman and Eric Van Douwen, and students such as Amer Be{\v s}lagi\'c , Zorana Lazarevi\'c and Frank Tall. The celebrated Dowker problem is an instance of research in this area. 

Let us denote by ${}^{<\kappa} \square X^\lambda$ the product of $\lambda$ many copies of a space $X$ given the topology generated by sets of the form
$\prod_{\alpha<\lambda} O_\alpha$ where each $O_\alpha$ is open in $X$ and only $<\kappa$ many $O_\alpha$ are not equal to the whole $X$.
Such a topology is called a $(<\kappa)$-box topology.
It is well known that the productivity of $(<\kappa)$-compactness in a $(<\kappa)$-box product of topological spaces leads to large cardinal notions:  a
cardinal $\kappa$ is strongly compact iff ${}^{<\kappa}\square 2^\lambda$ is $(<\kappa)$-compact for all $\lambda\ge\kappa$ and 
$\kappa$ is weakly compact iff ${}^{<\kappa}\square 2^\kappa$ is $(<\kappa)$-compact (see Theorem 2.10 and Theorem 5.1 in \cite{BuhagiarDzamonja1} for the proofs in the case of $\kappa=\kappa^{<\kappa}$ and Theorem \ref{singularnotcompact} here for the case when $\kappa$ is singular). One of the main questions we address in this paper is to what extent paracompactness of such products leads to a large cardinal notion.

In this paper, all spaces are Hausdorff. 

\subsection{Basic concepts and history}
For two covers $\OO$ and $\VV$ of a space $X$, we say that $\VV$ is a
{\em refinement} of $\OO$ if for every $V\in \VV$ there exists $O\in\OO$ such that $V\subseteq O$. The familiar definition of $X$ being compact, which says that every open cover has a finite subcover, can be stated in an equivalent way by saying that every open cover has a finite open refinement. Then it is clear that the following definition generalises 
compactness.

\begin{definition}\label{def:paracompact} A space $X$ is {\em paracompact} if every open cover of $X$ has a locally finite open refinement.
\end{definition}

By a {\em localy finite} family $\FF$ of sets in a topological space $X$ we mean a family such that for every $x\in X$ there exists an open set $O$ with $x\in O$ and such that $O$ intersects only finitely many sets in $\FF$.  Paracompactness was introduced by Jean Dieudonné in \cite{MR13297} who proved that every paracompact space is normal and that every 2nd countable metric space is paracompact. Moreover, he proved that paracompactness is of interest in analysis, since if $X$ is paracompact then for every pair of functions
$g,h:\,X\to \mathbb R$ such that $g$ is lower semi-continuous and $h$ is upper semi-continuous with $g(x)<h(x)$ for all $x$, there exists a continuous function $f:\,X\to \mathbb R$ satisfying  $g(x)<f(x)<h(x)$ for all $x\in X$. Dieudonné asked if every metrisable space was paracompact, and Ernest Michael answered in the affirmative in \cite{MR56905}. These results give an indication that paracompactness is close to metrisability, and indeed Smirnov's metrisation theorem states that a space $X$ is metrisable iff it is locally metrisable and paracompact (this theorem is a consequence of the Nagata-Smirnov metrisation theorem proved independently by Jun-iti Nagata in \cite{MR43448} and Yuriĭ Mikhaĭlovich Smirnov in \cite{MR0041420}).

Arthur H. Stone proved in \cite{MR26802} that  paracompactness  is equivalent to a strengthening of normality called ``full normality'',
introduced by John W. Tukey in \cite{MR0002515}. (Stone's result in fact already implies that metrisable spaces are paracompact). Just like normality, paracompactness is not necessarily preserved by products, as was already proved by
Dieudonné. The celebrated Dowker problem (resolved by M.E. Rudin in \cite{MR270328}, \cite{RudinDowker}, who found a ZFC example of a Dowker space)  connects these notions: a normal space $X$ is a Dowker space iff $X\times [0,1]$ is not normal iff $X$ is not countably paracompact. (These equivalences were obtained by Clifford H. Dowker in \cite{MR43446}). By {\em countably paracompact} we mean a space which satisfies the definition of
paracompactness when only countable covers are considered.

We hence have that it is interesting to consider even finite products of paracompact spaces. For infinite products, the Tikhonov product was introduced exactly in order to obtain that the product of compact spaces is compact, since the
more natural box topology on even the countably infinite product of the form $\prod_{n<\omega} X_n$ with the base
consisting of arbitrary products of open sets, fails to be compact even if every $X_n$ is compact. For example, the box product 
$\square\, 2^\omega$ is a discrete space of size $\mathfrak c$, hence not compact. However, such a product might be paracompact, as is the case of
$\square\, 2^\omega$. Kunen's result \cite{MR514975}, cited above, gives a characterisation under $CH$. Box products of the form $\square_{n<\omega} X_n$ became known as `countable box products' and were extensively studied by the Madison school. More on this concept can be found in the survey article \cite{MR3414877} by J. Roitman and Scott Williams.

\subsection{This paper}
In this paper we are interested in more general box products of the type $\prod_{\alpha<\kappa} X_\alpha$, but also in the `middle box' products, where the number of spaces in the product is larger than the defining cardinal of the boxes. For example, we may consider products of the
form $\prod_{\alpha<\omega_1} X_\alpha$ but with the base 
consisting of products $\prod_{\alpha<\omega_1} O_\alpha$ of open sets where a countable number of $O_\alpha$ is allowed to be different than $X_\alpha$. We are also interested in generalisatons of metrisability, since the experience with Smirnov's theorem has shown that metrisability is strongly connected with paracompactness. 

The main results of the paper are to give several sufficient conditions on the paracompactness of box products and some necessary ones. The former include results on generalised metrisability and Sikorski spaces. Of particular interest are products of the type ${}^{<\kappa} \square\,2^\lambda$, where we prove that for a regular uncountable cardinal $\kappa$, if ${}^{<\kappa}\square\, 2^\lambda$ is paracompact for every $\lambda\ge\kappa$, then $\kappa$ is at least inaccessible.
The case of the products of the type ${}^{<\kappa} \square\, X^\lambda$ for $\kappa$ singular has not been studied much in the literature and we offer various results.
The question if ${}^{<\kappa}\square \, 2^\lambda$ can be paracompact for all $\lambda$ when $\kappa$ is singular has been partially answered but
remains open in general.

The paper is organised as follows. In \S\ref{sec:intro-not} we give basic definitions and notation. In \S\ref{sec:omega+1} we recall the notion of
$\kappa$-metrisability, which implies paracompactness and we give sufficient conditions for it, some applications and some limitations. In \S\ref{sec:Sikorski} we recall a class of topological spaces studied by Sikorski, which we name Sikorski spaces and we prove a main Theorem \ref{big} which gives sufficient conditions for box products of Sikorski spaces to be paracompact. The short \S\ref{sec:scaling} gives a generalisation of a theorem of van Douwen, which shows that the existence of scales suffices for certain box products of Sikorski spaces to be paracompact. In \S\ref{sec:aleph-omega} we get to one of the main interests of the paper, the singular cardinals. The main theorem that we obtain here is Theorem \ref{th:gen-aleph_omega} which gives sufficient conditions for a ``tall-thin'' product of ordinal spaces to be
paracompact. \S \ref{sec:large_car} is a main section, where we define the notion of a paracompact cardinal and prove that for an uncountable regular cardinal being paracompact is a large cardinal notion. It remains unclear if this is also true in the case of singular cardinals. In \S\ref{sec:combo} we give a combinatorial characterisation of a cardinal being paracompact and we indicate some further research directions, relating to combinatorics and logic. Finally, \S\ref{ap:nabla} is an Appendix reviewing the Nabla product and those of its properties that we used in the previous sections of the paper.

\section{Notation and definitions}\label{sec:intro-not} We use $\kappa,\lambda$ and $\theta$ to denote infinite cardinals.
For a topological space $X$, we denote by $\tau(X)$ the topology of $X$, that is the family of all open subsets of
$X$.

A {\em neighbourhood base} of a point $x$ in a topological space $X$ is a family $\BB_x$ of open sets containing $x$ such that every
$O\in\tau(X)$, if $x\in O$ then there is $B\in \BB_x$ with $B\subseteq O$.

\begin{defn} Suppose that $\kappa\le\lambda^+$ are cardinals and $\langle X_\alpha: \,\alpha
<\lambda\rangle$ is a sequence of topological spaces. The $(<\kappa)${\em -box topology} on the set-theoretic
product $\prod_{\alpha<\lambda}X_\alpha$ has as a base, sets of
the form $\prod_{\alpha<\lambda}O_\alpha$, where each $O_\alpha$ is
open in $X_\alpha$ and $|\{\alpha <\lambda:\, O_\alpha\neq
X_\alpha\}|< \kappa$. This is denoted by ${}^{<\kappa}\,\bigbox_{\alpha<\lambda} X_\alpha$.

If all $X_\alpha$ are the same space $X$, we may write  ${}^{<\kappa}\,\bigbox X^\lambda$ in place of
${}^{<\kappa}\,\bigbox_{\alpha<\lambda} X_\alpha$.
\end{defn}

So, in the case of $\kappa=\aleph_0$, the $(<\kappa)$-box topology is
the usual Tikhonov topology, denoted by $\prod_{\alpha<\lambda}X_\alpha$. 
When $\kappa = \lambda^+$ we get the \emph{full $\lambda$-box product} denoted by 
$\bigbox_{\alpha<\lambda} X_\alpha$. When $\lambda=\aleph_0$, the full $\lambda$-box product is called, following 
\cite{MR3414877}, the {\em countable product}.

\begin{definition}\label{def:support} Suppose that $U=\prod_{\alpha<\lambda} U_\alpha$ is a basic open set in ${}^{<\kappa}\,\bigbox X^\lambda$.
We define the {\em support of} $U$ by $\supt(U)=\{\alpha<\lambda:\,U_\alpha\neq X_\alpha\}$. In the case that each $X$ is the discrete two point space
$2=\{0,1\}$, we identify $2^\lambda$ with the space ${}^\lambda 2$ of functions from $\lambda$ to $2$. Then the basic clopen sets in
 ${}^{<\kappa}\,\bigbox 2^\lambda$ are of the form $[s]=\{f\in 2^\lambda:\,s\subseteq f\}$ for $s$ a partial function from $\lambda\to 2$
 with $|\dom(s)|<\kappa$. In the terminology of Kunen's book [Def 6.1. VII \cite{kunen}] the family of all such $[s]$ is denoted 
 ${\rm Fn}(\lambda, 2,\kappa)$.
\end{definition}

We note that if ${}^{<\kappa}\,\bigbox X^\lambda$ is paracompact, then for every $\lambda'<\lambda$, the product ${}^{<\kappa}\,\bigbox X^{\lambda'}$, being
homeomorphic to a closed subspace of ${}^{<\kappa}\,\bigbox X^\lambda$, is also paracompact.
In many places in the paper we shall deal with a notion somewhat stronger than paracompactness, which we recall now.

\begin{definition}\label{ultraparacompact} A space $X$ is said to be {\em ultraparacompact} if every open cover of $X$ has a  disjoint clopen refinement.
\end{definition}

Clearly, a disjoint clopen cover is locally finite, so ultraparacompactness implies paracompactness. 

We mention the following classical fact which shows that paracompactness is preserved in products with compact spaces. 

\begin{lemma}\label{lem:compxpara} A product of a compact space $Y $ with a paracompact space $Z$ is paracompact. If $Y$ is compact 0-dimensional and
$Z$ is ultraparacompact, then $Y\times Z$ is ultraparacompact.
\end{lemma}

\begin{proof} Suppose that $Y$ is compact, $Z$ is paracompact and $\mathcal A$
is an open cover of $Y\times Z$. We shall argue as in the proof of the `Tube Lemma' of Tikhonov, as follows.

For any given $z\in Z$ and for each $y\in Y$ find $A\in \mathcal A$ and a basic open set $U^y\times V^y\subseteq A$ such that $(y,z)\in U^y\times V^y$.
Since $Y$ is compact, so is $Y\times \{z\}$, so we can find finitely many $\{y^z_0,\ldots y^z_{n_z-1}\}$ such that $Y\times \{z\}\subseteq \bigcup_{i<n_z}
U^{y_i}\times V^{y_i}$. Let $V_z=\bigcap_{i<n_z} V^{y_i}$, so $Y\times \{z\}\subseteq \bigcup_{i<n_z} U^{y_i}\times V_z$. For each
$i<n_z$ let $A^z_i\in \mathcal A$ be such that $U^{y_i}\times V^{y_i}\subseteq A^z_i$. Let ${\mathcal A}_z=\{A^z_i:\, i<n_z\}$.

The family $\{V_z:\;z\in Z\}$ is an open cover of $Z$, so by the 
paracompactness of $Z$ there is a locally finite open refinement $\BB=\{B_\alpha:\,\alpha<\alpha^\ast\}$ of
$\{V_z:\;z\in Z\}$. For each $\alpha<\alpha^\ast$, choose $z_\alpha\in Z$ such that $B_\alpha\subseteq V_{z_\alpha}$. For each $\alpha$ let
\[
\CC_\alpha=\{ A\cap (Y\times B_\alpha):\, A\in {\mathcal A}_{z_\alpha}\}
\]
and let $\CC=\bigcup_{\alpha<\alpha^\ast} \CC_\alpha$. Clearly, $\CC$ consists of open sets each of which is a subset of an element of $\mathcal A$.
We claim that $\CC$ covers $Y\times Z$. Indeed, if $(y,z)\in Y\times Z$, then first of all there is $\alpha$ such that $z\in B_\alpha$, but then 
$\CC_\alpha$ covers $Y\times B_\alpha$ and hence $(y,z)$ is in some element of $\CC_\alpha$. Hence $\CC$ is an open refinement of
$\mathcal A$ that is a cover.

In the case of paracompactness, we can finish by observing that $\CC$ is locally finite.  

In the case of the ultraparacompactness of $Z$ and 0-dimensionality of $Y$, by passing to a subcover if necessary, we may assume that $\mathcal A$ consists of clopen sets and that $\BB$ consists of disjoint clopen sets. Then each $\CC_\alpha$ is a finite family of clopen sets and
for $\alpha\neq\beta$ the sets in $\CC_\alpha$ are disjoint from the sets in $\CC_\beta$ . We are not
done yet, since it may happen that the sets within the same $\CC_\alpha$ are not pairwise disjoint.

However, each $\CC_\alpha$ consists of a finite family of clopen sets, so we can clearly refine $\CC_\alpha$ to another finite family of pairwise disjoint clopen sets, which has the same union.
$\eop_{\ref{lem:compxpara}}$
\end{proof}

We shall be using the following concepts from infinite combinatorics:

\begin{definition}\label{def:kappaopencard} (1) The cardinal $\mathfrak{b}(\kappa)$ is the smallest size of an unbounded family in the structure $({}^{\kappa}\kappa, \le^\ast_\kappa)$ of the functions from $\kappa$ to $\kappa$ ordered by 
\[
f\le^\ast_\kappa g\iff |\{\alpha <\kappa:\,f(\alpha)>g(\alpha)\}|<\kappa.
\]
We write $f<^\ast_\kappa g$ if $f\le^\ast_\kappa g$ and not $g\le^\ast_\kappa f$.

\smallskip

{\noindent (2)} A {\em scale} in $({}^{\kappa}\kappa, \le^\ast_\kappa)$ is a sequence of the form $\langle f_\alpha:\,\alpha<\lambda\rangle$ such that
\[
\alpha<\beta\implies f_\alpha<^\ast_\kappa f_\beta
\]
and for every $g\in {}^{\kappa}\kappa$ there is $\alpha<\lambda$ such that $g\le^\ast_\kappa f_\alpha$.
\end{definition}

\section{Generalised metrisability and a short review of $\square (\omega+1)^\lambda$ and $\square (\omega_1+1)^{\lambda}$ }\label{sec:omega+1} 
Although our focus is not on countable box products nor those with $\omega_1$ many factors, we start this section with some of the main
results obtained about such products, which are the case most studied in the literature. These results will serve as a motivation and also as a way to introduce certain techniques which will be used later in a more general setting. The main theorem in this section is Theorem \ref{P:G2}, of which part (i) was stated by Van Douwen and part (ii) is due to us.

In spite of many partial results, one of the most basic questions about box product remains open:

\begin{question}\label{open} Is the product $\square (\omega+1)^\omega$ normal?
\end{question}

L. Brian Lawrence proved in \cite{MR1303123} that $\square (\omega+1)^{\omega_1}$ is not normal. A positive answer to Question \ref{open} is known under various set-theoretic assumptions, see Roitman-Williams \cite{MR3414877}.
Miller \cite{MR675593}
considered the middle box product and proved the following theorem under CH.

\begin{thm}[A. Miller]\label{T:M82} $[\rm CH]$ If for each $\alpha < \omega_1$, $X_\alpha$ is a compact 
Hausdorff space of weight $\leq \mathfrak{c}$, then ${}^{<\omega_1}\!\bigbox_{\alpha<\omega_1} 
X_\alpha$ is paracompact.
\end{thm}

In particular, under $CH$, the product ${}^{<\omega_1}\square (\omega+1)^{\omega_1}$ is paracompact and hence normal.
The next natural case to consider is the space obtained by replacing $\omega+1$ with the space $\omega_1+1$ obtained by 
having the discrete topology on $\omega_1$ and letting open neighbourhoods of the point $\omega_1$ be sets 
of the form $O\cup \{\omega_1\}$ where $O$ is a co-countable subset of $\omega_1$. This space is not compact, so considering its box products presents new challenges. The first theorem we mention is
due to M.E. Rudin in \cite{MR0407794}. 

\begin{theorem}[M.E. Rudin]\label{MERudinth}$\square (\omega_1+1)^{\omega}$ is ultraparacompact.
\end{theorem}

In Theorem \ref{E:D+1a} we shall prove van Douwen's theorem that ${}^{<\omega_1}\bigbox (\omega_1+1)^{\omega_1}$
is ultraparacompact. A consequence of our Theorem \ref{big}(b) is that under $CH+{\mathfrak b}(\omega_1)=2^{\omega_1}$, the full product 
$\bigbox (\omega_1+1)^{\omega_1}$ is ultraparacompact. 

On the way to Theorem \ref{E:D+1a}, we need to develop some machinery. The following definition can be found in 
van Douwen's article \cite{MR564100}, pg. 90.\footnote{The original definition of $\kappa$-metrisable spaces is due to Roman Sikorski in \cite{SikorskiDelta} and is somewhat stronger than the one that we use. It can be proved that the two definitions agree for $\kappa$ of uncountable cofinailty. Other definitions equivalent to ours also exist in the literature.}

\begin{definition}\label{def:kappametrisable} A space $X$ is {\em $\kappa$-metrisable} if there is a function 
\[
U: X \times \kappa \to \tau(X)
\]
such that
\begin{enumerate}[(1)]
\item $\{U(x,\alpha):\alpha < \kappa\}$ is a neighbourhood base at $x$ and
\item for all $x,y \in X$  and $\alpha\le \beta < \kappa$,
\begin{enumerate}[(a)]
\item if $y \in U(x,\alpha)$ then $U(y,\beta) \subseteq U(x,\alpha)$, and
\item if $y \notin U(x,\alpha)$ then $U(y,\beta) \cap U(x,\alpha) = \emptyset$.
\end{enumerate}
\end{enumerate}
\end{definition}

The property (2)(a) in particular implies that if $x\in X$ and $\alpha\le\beta$, then $U(x,\beta)\subseteq U(x,\alpha)$.
The interest of this generalised metrisability is that it implies ultraparacompactness, as shown in 
\cite{MR564100} pg. 90 Claim 2. We give a variation of the proof, since we shall need it in the proof of Theorem \ref{boundedmetrisable}.

\begin{example}\label{omegameteic} Recall that  an ultrametric $d^*$ on a topological space $X$ is a metric that in addition satisfies 
that for all $x,y,z\in X$
\[
d^*(x,y)\leq \max\{d^*(x,z),d^*(y,z)\}\,.
\]
The $\omega$-metrisable spaces are precisely the ultrametric spaces. Namely, suppose that $d^\ast$ is an ultrametric on the space $X$ and
define for every $x\in X$ that $U(x,0)=X$ and $U(x, n)=B_{d^\ast}(x,1/n)$, for $n>0$. To verify the condition (2)(b), we note that if $y\notin U(x, n)$,
then $n>0$ and $d^\ast(x,y)\ge 1/n$. If there were to be some $z$ such that $d^\ast(z,y)<1/n$ and $d^\ast(x,z)< 1/n$, then the property of ultrametric would imply that $d^\ast(y,x)< 1/n$, a contradiction. So $U(x,n)\cap U(y,m)=\emptyset$ for all $m\ge n$.

In the other direction, suppose that $X$ is $\omega$-metrisable. Without loss of generality we assume that for all $x\in X$ we have $U(x,0)=X$. Let us define for $x\neq y$ the value $d^\ast(x,y)=\frac{1}{\min\{n:\,U(x,n)\cap U(y,n)=\emptyset\}}$, which is well defined by property (1) of  $\omega$-metrisability. 
The proof that this is an ultrametric is a backwards run of the proof in the previous paragraph.

A concrete example of an $\omega$-metrisable space is the binary tree $2^{\omega}$ where $U(\rho,n)$ for $\rho\in 2^{\omega}$ consists of all $\sigma$ which have the same first $n$ digits as $\rho$.
\end{example}

\begin{claim}\label{vanDouwenCl2} Every $\kappa$-metrisable space $X$ is ultraparacompact, and moreover, if the sets $U(x,\alpha)$ are
basic (cl)open, then the refinements witnessing ultraparacompactness can always be chosen to be basic (cl)open sets.
\end{claim}

\begin{proof} Suppose that $U$ is a function witnessing the $\kappa$-metrisability of $X$. Note that for all $x,y\in X$ and $\alpha<\kappa$, if 
$y\in U(x,\alpha)$, then on the one hand $U(y,\alpha)\subseteq U(x,\alpha)$, by (2)(a), but then by (2)(b) also $x\in U(y,\alpha)$ and hence
$U(x,\alpha)=U(y,\alpha)$. Let $\OO$ be any open cover of $X$.

We can define $\rho:\,X\to\kappa$ by letting 
\[
\rho(x)=\min\{\alpha<\kappa:\,(\exists O\in \OO) \,U(x,\alpha)\subseteq O\}.
\]
Then $\{U(x,\rho(x)):\, x\in X\}$ is a disjoint refinement of $\OO$ (consisting of basic (cl)open sets if all $U(x,\alpha)$ are basic (cl)open).
$\eop_{\ref{vanDouwenCl2}}$
\end{proof}

We now prove the following Theorem \ref{P:G2} which gives sufficient conditions for $\kappa$-metrisability. Part (i) is stated as Fact 1 in van Douwen's article \cite{VANDOUWEN197771}.

\begin{theorem}\label{P:G2} Suppose that $\kappa,\lambda,\theta$ are infinite cardinals such that 
\begin{enumerate}
\item[(i)] either $\kappa=\lambda=\theta$ is regular or 
\item[(ii)] $\kappa=\lambda^+$  and \footnote{In this case ${}^{<\kappa}\bigbox_{\alpha<\lambda} X_\alpha=\bigbox_{\alpha<\lambda} X_\alpha$.} $\cf(\theta)>\lambda$.
\end{enumerate}
Then if $\langle X_\alpha: \,\alpha < \lambda\rangle$ is a sequence of $\theta$-metrisable spaces, 
then ${}^{<\kappa}\bigbox_{\alpha<\lambda} X_\alpha$ is also $\theta$-metrisable, and therefore ultraparacompact.
\end{theorem}

\begin{proof} Let $X={}^{<\kappa}\bigbox_{\alpha<\lambda} X_\alpha$. For every $\alpha <\lambda$, let 
$U_\alpha: X_\alpha \times \theta \to \tau(X_\alpha)$
be a function witnessing the $\theta$-metrisability of $X_\alpha$.
We define
$W: X \times \theta \to \tau(X)$
by
\[
W(x, i) = \prod_{\alpha < \lambda}W^\alpha_i (x_\alpha)\,,
\]
where for $\alpha<\lambda$ and $i<\kappa$ we have
\[
W^\alpha_i (x_\alpha) = \begin{cases} 
U_\alpha(x_\alpha,i) & \text{if }\alpha \le i\, \\
X_\alpha & \text{ if } i<\alpha 
\end{cases}
\]
and for $i\ge\kappa$, in the case $\cf(\theta) >\lambda$, we define $W(x, i) = \prod_{\alpha < \lambda} U_\alpha(x_\alpha,i)$, 
for any $x = (x_\alpha)_{\alpha<\lambda}\in X$. We shall check that $W$ witnesses that $X$ is a
$\theta$-metrisable space. 

Since $W^\alpha_i (x_\alpha)\neq X_\alpha$ implies that $\alpha\le i <\kappa$, we obtain that 
$W(x,i)$ is open in $X$. Let $x = (x_\alpha)_{\alpha<\lambda}\in X$ and let $O$ be a basic open set
containing $x$. Hence $O= \prod_{\alpha < \lambda}O_\alpha$ where for each $\alpha$ we have that $x_\alpha\in O_\alpha$
and that $S\deq\{\alpha<\lambda:\, O_\alpha\neq X_\alpha\}$ is a set of size $<\kappa$. For each 
$\alpha<\lambda$ choose
an $i_\alpha<\theta$ such that $U_\alpha(x_\alpha, i_\alpha)\subseteq O_\alpha$. 

(i) In the case that $\kappa=\lambda=\theta$ is regular, we have that each $i_\alpha<\kappa$, so letting $i^\ast=\sup_{\alpha\in S} i_\alpha$,
by the regularity of $\kappa$ we have that $i^\ast<\kappa=\theta$.
For every $\alpha\in S$ we have $U_\alpha(x_\alpha, i^\ast)$ is well defined and $U_\alpha(x_\alpha, i^\ast)\subseteq U_\alpha(x_\alpha, i_\alpha)\subseteq O_\alpha$. Since for $\alpha\notin S$ we have $O_\alpha=X_\alpha$, we can conclude 
$U_\alpha(x_\alpha, i^\ast)\subseteq O_\alpha$ for every $\alpha<\lambda$. Hence $W(x,i^\ast)\subseteq O$.

(ii) In the case that $\kappa=\lambda^+$, on the one hand, $\cf(\theta)>\lambda$, so $i^\ast<\theta$ and for every
$\alpha<\lambda$, we have that $U_\alpha(x_\alpha, i^\ast)$ is well defined. On the other hand,
we are dealing with the full box topology, so we have that $W(x, i^\ast)=\prod_{\alpha < \lambda} U_\alpha(x_\alpha, i^\ast)
\subseteq O$. 

We have hence proved the property (1) from Definition \ref{def:kappametrisable}, since it suffices to verify it for basic open sets.

To prove (2)(a), suppose that $x,y \in X$ and $y\in W(x,i)$ for some $i<\kappa$. This means that for every $\alpha \le i$ we have 
$y_\alpha \in U_\alpha(x_\alpha,i)$, so for all $j\ge i$ we have $U_\alpha(y_\alpha,j)\subseteq U_\alpha(x_\alpha,i)$.
On the other hand, in the case $\kappa=\lambda$, if $j\ge i$ and $\alpha > i$ then we have that $W^\alpha_i(x_\alpha)=X_\alpha$ so certainly 
$W^\alpha_j(y_\alpha)\subseteq W^\alpha_i(x_\alpha)$. If $\kappa=\lambda^+$ and $i\ge\lambda$ we argue similarly to the case $i<\kappa$. Hence $W(y,j)\subseteq W(x,i)$.

Finally for (2)(b), if $x,y\in X$ and  $y\notin W(x,i)$ for some $i<\kappa$, then there exists some $\alpha\le i$ such that $y_\alpha\notin 
U_\alpha(x_\alpha,i)$. Therefore for all $j\ge i$ we have that $U_\alpha(y_\alpha, j)\cap U_\alpha(x_\alpha, i)=\emptyset$ and so for all
$j\ge i$ we have $U(y,j)\cap U(x,i)=\emptyset$. In the case $\kappa=\lambda^+$ and $i\ge\kappa$ the argument is similar.
$\eop_{\ref{P:G2}}$
\end{proof}

A direct consequence of Theorem \ref{P:G2}, applied to the cases $\kappa=\omega_1=\lambda$ is the following.

\begin{cor}\label{C:G1}
Suppose that $\langle X_\alpha: \,\alpha < \omega_1\rangle$ is a sequence of $\omega_1$-metrisable 
topological spaces. 
Then ${}^{<\omega_1}\!\bigbox_{\alpha<\omega_1} X_\alpha$ is also $\omega_1$-metrisable, and therefore ultraparacompact.
\end{cor}

To obtain Theorem \ref{E:D+1a} , we need to make the following easy observation.

\begin{claim}\label{omega_1metrisable} The space $\omega_1+1$ is $\omega_1$-metrisable.
\end{claim}

\begin{proof} For $\zeta\le\omega_1$ and $\alpha<\omega_1$ let 
\[
U(\zeta, \alpha)=
\begin{cases}
\{\zeta	\}&\mbox{if }\zeta<\alpha\\
[\alpha,\omega_1] &\mbox{otherwise}.
\end{cases}
\]
$\eop_{\ref{omega_1metrisable}}$
\end{proof}

As a direct consequence of Corollary \ref{C:G1}, we finally obtain the following theorem of van Douwen (discussion on top of 75 in
\cite{VANDOUWEN197771}):

\begin{theorem}[E. van Douwen]\label{E:D+1a} 
The space ${}^{<\omega_1}\bigbox (\omega_1+1)^{\omega_1}$ is $\omega_1$-metrisable and hence ultraparacompact.
\end{theorem}

Of course, the same methods yield the ultraparacompactness of spaces of the form 
${}^{<\lambda}\bigbox (\lambda+1)^{\lambda}$ where $\lambda$ is a regular cardinal and 
$\lambda+1$ is given the topology in which $\lambda$ is discrete and the point $\lambda$ has open neighbourhoods of the form $A\cup\{\lambda\}$ where $A$ is a final segment of $\lambda$. Similarly for other variations fitting the cardinal requirements of Theorem \ref{P:G2}. 

One may wonder if this method can yield the solution of Question \ref{open}, applying it with $\kappa=\omega_1$ and
$\lambda=\omega$. Unfortunately, this does not work as $\omega+1$ is only $\omega$-metrisable, not 
$\omega_1$-metrisable.  Examples when the normality fails were obtained by Arthur Stone \cite{MR26802}, Carlos Borges \cite{MR253274} and Bankston \cite{MR0454898}, see \S3 of  \cite{MR0454898}. Similarly, Theorem \ref{P:G2} does not generalise to the situation $\kappa<\lambda$, as shown by the following.

\begin{theorem}\label{ex:negative} Suppose that $\lambda\ge\omega_2$ and $\langle X_\alpha:\,\alpha<\lambda\rangle$ are spaces where $S=\{\alpha <\lambda:\,|X_\alpha|\ge 2\}$ has size at least $\aleph_2$. Then the space 
$X={}^{<\omega_1}\prod_{\alpha<\lambda} X_\alpha$ is not normal (and hence certainly not paracompact or ultraparacompact).
\end{theorem}

\begin{proof} We first notice that $X$ has a closed subspace a product of the form 
\[
F_\alpha=
\begin{cases}
\{x_\alpha, y_\alpha\}\mbox{ if }\alpha\in S\\
\{x_\alpha\}\mbox{ otherwise},
\end{cases}
\]
(for some $x, y\in X$) which contains a closed homeomorphic copy of ${}^{<\omega_1}\{0,1\}^{\omega_2}$. The latter space is not normal, as shown by
van Douwen, Theorem B of \cite{VANDOUWEN197771}. As the normality is inherited by closed spaces, we conclude that $X$ is not normal.
$\eop_{\ref{ex:negative}}$
\end{proof}

We shall discuss further examples of non-paracompactness of box products in \S\ref{sec:large_car}.

\section{Sikorski spaces and Kunen's and Miller's Theorems for uncountable $\kappa$}\label{sec:Sikorski}
As mentioned in \S\ref{sec:intro}, one of the main motivations of studying the paracompactness of box products is that, knowing that such products are not compact for general $\kappa$ even when the factors consist of the two element discrete space, they however sometimes turn out to be paracompact. In this section we shall endow our individual spaces with some element of compactness and then study the paracompactness of the $\kappa$-box product. 
Let us state the definition of the level of compactness we need.

\begin{definition}\label{<kappacompact} Let $\kappa$ be an infinite cardinal.  A topological space $X$ is said to be 
{\em $(<\kappa)$-compact} if every open cover of $X$ has an open subcover of size $<\kappa$.
\end{definition}

This notion is well studied in set theory, in model theory and in topology and in particular gives a well known characterisation of  strong compactness of a cardinal $\kappa$: an uncountable $\kappa$ is strongly compact iff ${}^{<\kappa}\square\, 2^\lambda$ is $(<\kappa)$-compact for any $\lambda\ge\kappa$ iff for any $(<\kappa)$-compact $X$ the product ${}^{<\kappa}\square X^\lambda$ is $(<\kappa)$-compact, for every $\lambda\ge\kappa$ (for a proof see Theorem 2.10 in \cite{BuhagiarDzamonja1}). In view of this result, it is natural to ask in which circumstances the product of the form ${}^{<\kappa}\square X^\lambda$
for a $(<\kappa)$-compact $X$, is paracompact. In Section \ref{sec:large_car} we shall be specifically interested in that question when $X$ is simply the point 
space 2.

We shall work in the class of regular spaces and will use an idea of Sikorski \cite{SikorskiDelta}, who was the first one to consider spaces called $(<\kappa)$-open here (and which Sikorski called $\kappa$-additive). 

\begin{definition}\label{def:kappaopen} A space $X$ is {\em $(<\kappa)$-open} if the intersection of any family of
$<\kappa$ many open sets in $X$, is open.
\end{definition}

In the presence of 
$(<\kappa)$-compactness, the class of regular $(<\kappa)$-open spaces behaves very nicely with respect to products, which was exploited by Kunen in 
\cite{MR514975} and which 
we shall use to our advantage to prove Theorem \ref{big} below. In honour of Sikorski, we introduce the following definition.

\begin{definition}\label{def:Sikorskispace} A {\em $\kappa$-Sikorski space} is a regular $(<\kappa)$-compact $(<\kappa)$-open space.
\end{definition}

The following is the main theorem of this section.

\begin{theorem}\label{big} Suppose that $\kappa=\kappa^{<\kappa}$ and $\langle X_\alpha:\,\alpha<\kappa^+\ra$ is a sequence of $\kappa$-Sikorski 
spaces each of weight $\le\kappa$. Then:
\begin{enumerate}[{\rm (a)}]
\item ${}^{<\kappa}\bigbox_{\alpha<\kappa} X_\alpha$ is paracompact.
\item If ${\mathfrak b}(\kappa)=\kappa^+$, then $\bigbox_{\alpha<\kappa} X_\alpha$ is paracompact.
\item If $2^\kappa=\kappa^+$, then ${}^{<\kappa^+}\bigbox_{\alpha<\kappa^+} X_\alpha$ is paracompact.
\end{enumerate}
\end{theorem}

Theorem \ref{big}(b) with 
$\kappa=\aleph_0$ gives the forward direction of Kunen's Theorem and Theorem \ref{big}(c) with $\kappa=\aleph_0$ gives the instance of Miller's Theorem for compact metric spaces.
The proof of Theorem \ref{big} uses Kunen's and Miller's technology of $\nabla$-products, which is reviewed in the Appendix \S\ref{ap:nabla},
with some new elements that we mention here. We also take a moment to recall some facts about Sikorski spaces that will be needed in the proof.

It is rather immediate to obtain Sikorski's theorem \cite{SikorskiDelta}(iv), which states that for any uncountable 
$\kappa$, every regular $(<\kappa)$-open space is 0-dimensional. In fact, more is true, as shown in the following Lemma \ref{obs:0dim}
proven by Arvind Misra in \cite{Misra}. His proof uses the Stone-\v{C}ech compactification, but a direct proof is rather easy. 

\begin{lemma}\label{obs:0dim} Suppose that $\kappa\ge \aleph_1$ and that $X$ is $(<\kappa)$-open. Then $X$ is strongly zero-dimensional, namely every two disjoint zero sets can be separated by clopen sets.
\end{lemma}

\begin{proof} Since being $(<\kappa)$-open implies being $(<\aleph_1)$-open, it suffices to show the claim in the case of $\kappa=\aleph_1$.
Let $A,B\subseteq X$ be functionally separated subsets of $X$ and let $f:X\to [0,1]$ be a continuous map such that $f[A] = \{0\}$ and $f[B]=\{1\}$. Let $U = f^{-1}[[0,\frac12]]$. Then\footnote{As the previous notation is somewhat cumbersome, from now on we are less careful in making a distinction between the round and the square brackets and will prefer to write things such as $f^{-1}([0,\frac12])$} $U$ is a 
closed $G_\delta$-set, and hence clopen, since $X$ is $\aleph_1$-open.
$\eop_{\ref{obs:0dim}}$
\end{proof}

In a strongly zero-dimensional space, paracompactness coincides with the stronger property of ultraparacompactness, defined in
Definition \ref{ultraparacompact} (see for example Theorem 10 in \cite{vanName} for a proof).
We can conclude that:

\begin{cor}\label{C:PUP} Suppose that $\kappa\ge\aleph_1$. Then any regular $(<\kappa)$-open space $X$ is paracompact iff it is ultraparacompact.
\end{cor}

This nicely fits with the following lemma, which in particular implies that Sikorski spaces are automatically
paracompact. The lemma was found independently by several people, including Bankston [Lemma 2.1. in \cite{MR0454898}] and Kunen [Lemma 1.3 in \cite{MR514975}]. It also follows as a corollary from Lemma \ref{T:GMK} below.

\begin{lemma}\label{notchtheorem} Suppose that $X$ is regular $(<\kappa)$-open and $(< \kappa^+)$-compact. Then $X$ is paracompact (hence ultraparacompact by Corollary \ref{C:PUP}).
\end{lemma}

As far as the products of spaces are concerned, the beginning is to handle finite products of $\kappa$-Sikorski spaces. This can be done using a representation theorem from \cite{SikorskiDelta}(xvii), stated as follows.

\begin{theorem}[Sikorski representation theorem]\label{thm:representation} Suppose that $\BB$ is a $(<\kappa)$-complete Boolean algebra. Then the following are equivalent:

\begin{enumerate}
\item\label{cond1} any $(<\kappa)$-complete filter of elements of $\BB$ extends to a $(<\kappa)$-complete ultrafilter.
\item $\BB$ is isomorphic to the field of clopen subsets of a $\kappa$-Sikorski space.
\end{enumerate}
\end{theorem}

\begin{corollary}\label{cor:TubeLemma} The product of two $\kappa$-Sikorski spaces is $\kappa$-Sikorski.
\end{corollary}

\begin{proof} Let $X_0$ and $X_1$ be two $\kappa$-Sikorski spaces and let $\BB_l$ be the Boolean algebra of clopen subsets of $X_l$ for $l<2$. Hence each
$\BB$ is $(<\kappa)$-complete and satisfies the condition \ref{cond1}. of Theorem \ref{thm:representation}.
The algebra of clopen sets of $X_0\times X_1$ is exactly the algebra $\BB_0\times \BB_1$, hence by the Stone representation theorem, $X_0\times X_1$ is exactly the unique up to homeomorphism Stone space of that algebra.
Since $\BB_0\times \BB_1$ is evidently $(<\kappa)$-complete, to prove that $X_0\times X_1$ is $\kappa$-Sikorski, it suffices to prove that every 
$(<\kappa)$-complete filter of elements of $\BB_0\times \BB_1$ extends to a $(<\kappa)$-complete ultrafilter.

Let $\FF$ be a $(<\kappa)$-complete filter of elements of $\BB_0\times \BB_1$ and for $l<2$ let $\FF_l$
be the projection of $\FF$ to $\BB_l$, that is 
\[
\FF_0=\{a_0\in \BB_0:\, (\exists a_1\in \BB_1)(a_0, a_1)\in \FF\},
\]
and similarly for $\FF_1$. It is easy to check that each $\FF_l$ is a $(<\kappa)$-complete filter on $\BB_l$, and
hence can be extended to a $(<\kappa)$-complete ultrafilter $\UU_l$ on $\BB_l$. Clearly, $\UU_0\times \UU_1$ is
a $(<\kappa)$-complete filter on $\BB$ and it extends $\FF$. We only need to prove that it is an ultrafilter. 

Suppose to the contrary, that $\GG$ is a filter on $\BB$ which properly extends $\UU_0\times \UU_1$ and let 
$(a_0, a_1)\in \GG\setminus (\UU_0\times \UU_1)$. Then for some $l<2$, $a_l\notin \UU_l$, say $a_0\notin\UU_0$.
Therefore $-a_0\in \UU_0$ and so $a_0, -a_0$ are both in the projection $\GG_0$ of $\GG$ to $\BB_0$. This is 
a contradiction since $\GG$ being a filter implies that $\GG_0$ is a filter.
$\eop_{\ref{cor:TubeLemma}}$
\end{proof}

In fact, even without the regularity, being $(<\kappa)$-open $(<\kappa)$-compact is preserved by finite products. This is because 
of the following version of the `Tube Lemma' of Tikhonov:

\begin{lemma}\label{TubeL} If $C$ is a ($<\kappa$)-compact subspace of a Hausdorff $(<\kappa)$-open space $X$ 
and $y$ a point of a space $Y$, then for every open set $W\subseteq X\times Y$ containing 
$C\times\{y\}$, there exists open sets $U\subseteq X$ and $V\subseteq Y$ such that 
$C\times\{y\} \subseteq U\times V \subseteq W$.
\end{lemma}

The Lemma is proved along the same lines as that for compact spaces, see for example \cite[Lemma 3.1.15]{E}.
From this lemma one can easily obtain the following two results:

\begin{observation}\label{P:M2.1d} If a space $X$ is a Hausdorff, ($<\kappa$)-compact $(<\kappa)$-open space then for any space $Y$ the projection $\pr_Y: X\times Y \to Y$ is a closed map.
\end{observation}

\begin{observation}\label{P:M2.2d} The product of two ($<\kappa$)-compact $(<\kappa)$-open spaces is ($<\kappa$)-compact $(<\kappa)$-open.
\end{observation}

In the following we use the Nabla spaces. For definitions and basic properties see Definition \ref{nabla} and \S\ref{subsec:tau} in Appendix \ref{ap:nabla}.

\begin{lemma}\label{C:6.12d} If for every 
$\alpha<\kappa$ the space $X_\alpha$ is a ($<\kappa$)-compact and $(<\kappa)$-open space, then the quotient map $q: \square_{\alpha<\kappa} X_\alpha \to 
\nabla_{\alpha<\kappa} X_\alpha$ is closed.
\end{lemma}

\begin{proof} We need to show that whenever $K$ is closed in $X=\square_{\alpha<\kappa} X_\alpha$, so is $q(K)$ in $\nabla_{\alpha<\kappa} X_\alpha$.
The latter, by the definition of the
quotient topology, happens exactly when $q^{-1}q(K)$ is closed in $X$. We have
\[
q^{-1}q(K)=\bigcup_{\alpha < \kappa} \tau_\alpha^{-1}\tau_\alpha (K).
\]
Hence, if we show that every
$\tau_\alpha (K)$ is closed in $\bigbox_{\alpha\leq\beta <\kappa} X_\beta$, then by Lemma 
\ref{L:6.10d}, the subspace $q^{-1}q(K)$ is closed. It suffices to show the following claim.

\begin{claim}\label{cl:closedmap} For every subset $A\subseteq \kappa$ with $|A|<\kappa$, the natural projection
$\pr_{\kappa\setminus A}: \bigbox_{\beta<\kappa} X_\beta \to \bigbox_{\beta\in\kappa\setminus A} X_\beta$ is a closed map.
\end{claim}

\begin{proof} The proof is by induction on $|A|$.
If $A$ is finite then the fact that $\pr_{\kappa\setminus A}$ is closed follows from Proposition \ref{P:M2.1d}.

Now assume that $A$ is infinite and that $\pr_{\kappa\setminus B} (K)$ is closed in 
$\bigbox_{\beta\in\kappa\setminus B} X_\beta$ for every $B\subseteq A$ with $|B| < |A|$. Let $\langle F_\gamma:\,\gamma< |A|\rangle$ be an increasing sequence of sets with $|F_\gamma|<|A|$ and such that $\bigcup_{\gamma< |A|} F_\gamma = A$. By assumption, $\pr_{\kappa\setminus F_\gamma}$ is a closed map and thus, $\pr_{\kappa\setminus F_\gamma}(K)$ is a closed set in $\bigbox_{\beta\in\kappa\setminus F_\gamma} X_\beta$, for every $\gamma<|A|$.

Choose an arbitrary point $z\in X$ and let
$K_\gamma = \prod_{\beta\in F_\gamma}\{z_\beta\} \times \pr_{\kappa\setminus F_\gamma}(K)$, which is closed in 
$X$. Then $\bigcap_{\gamma<|A|}K_\gamma =  
\prod_{\beta\in A}\{z_\beta\} \times \pr_{\kappa\setminus A} (K)$ is also closed in 
$X$. Hence, $\pr_{\kappa\setminus A} (K)$ is closed in 
$\bigbox_{\beta \in\kappa\setminus A} X_\beta$, as required.
$\eop_{\ref{cl:closedmap}}$
$\eop_{\ref{C:6.12d}}$
\end{proof}
\end{proof}

The main interest of Lemma \ref{C:6.12d} is that under certain circumstances it allows us to lift the paracompactness of the Nabla space to the
whole box product. This is done with the help of the following Lemma \ref{L:6.8d}. First we need a theorem that for $\kappa=\aleph_0$ was proven by Ernst Michael as Theorem 1 in \cite{MR87079}. The proof was extended for 
$\kappa=\aleph_1$ by Misra as Theorem 4.3 in \cite{Misra}. Slight modifications give the case we need.

\begin{defn} Let $X$ be a space. 
\begin{itemize}
\item A family $\U$ of subsets of $X$ is said to be {\em locally ($<\kappa$)} if each point $x\in X$ has a neighbourhood $O$ such that $|\{ U\in\U:\,O\cap U\neq\emptyset\}|<\kappa$.
\item $X$ is said to be {\em para-($<\kappa$)-compact}, if every open cover of $X$ has a locally ($<\kappa$) open refinement.
\end{itemize}
\end{defn}

\begin{lemma}\label{T:GMK} Let $\kappa$ be an infinite regular cardinal. For a regular ($<\kappa$)-open space $X$ the following are equivalent:
\begin{enumerate}[{\rm (1)}]
\item $X$ is paracompact,
\item $X$ is para-($<\kappa$)-compact,
\item every open cover of $X$ has a refinement of the form $\bigcup_{\alpha < \kappa}\U_\alpha$, 
where each $\U_\alpha$ is a locally ($<\kappa$)-collection of open sets.
\end{enumerate}
\end{lemma}

We shall use the following theorem of Michael, given as Theorem 1 in \cite{MR87079}.

\begin{theorem}[E. Michael]\label{MichaelTh1} A regular space $X$ is paracompact iff every open cover of $X$ has a refinement $\WW$
which is {\em closure preserving}, in the sense that for every $\WW'\subseteq \WW$ we have $\overline{\bigcup \WW'}=\bigcup \{\overline{W}:\,W\in \WW'\}$.
\end{theorem}

\begin{proof} (of Lemma \ref{T:GMK}.) The implications (1) $\implies$ (2) $\implies$ (3) are obvious.

\smallskip 

\noindent{(3) $\implies$ (2)} Let $\OO$ be an open cover of $X$. By (3), there exists a refinement of $\OO$ of the form
$\bigcup_{\alpha < \kappa}\U_\alpha$, where each $\U_\alpha$ is a locally ($<\kappa$) family of open sets. For each $\alpha < \kappa$, 
let $V_\alpha = \bigcup\U_\alpha$. Then $\VV = \{V_\alpha:\alpha < \kappa\}$ is an open cover of $X$. For each $\alpha<\kappa$,
let
\[
W_\alpha = V_\alpha \setminus \bigcup_{\beta < \alpha}V_\beta \,.
\]
Then we claim that $\W = \{W_\alpha:\alpha < \kappa\}$ is a locally ($<\kappa$) refinement of $\VV$. Indeed, given $x\in X$, let 
$\alpha(x) = \min\{\alpha:\,x\in V_\alpha\}$. Then $V_{\alpha(x)}$ intersects $<\kappa$ many members of $\W$, only those whose index is $<\alpha(x)$. 

Finally, the collection $\W' = \{W_\alpha \cap U: \alpha < \kappa, U\in\U_\alpha\}$ is 
a locally ($<\kappa$) refinement of $\OO$. Indeed, let $x\in X$ and $\alpha(x)$ be as above. For every $\alpha\leq\alpha(x)$, there is a neighbourhood $G_\alpha$ of $x$ that intersects $<\kappa$ many members of $\U_\alpha$. Let $G = V_{\alpha(x)} \cap \bigcap_{\alpha\leq\alpha(x)} G_\alpha$, which is open, since $X$ is ($<\kappa$)-open, and contains $x$. 
Then $G$ intersects $<\kappa$ many members of $\W'$ since by construction we have that $G\cap W_\alpha=\emptyset$ for $\alpha>\alpha(x)$.

\smallskip 

\noindent{(2) $\implies$ (1)} Let $\OO$ be an open cover of $X$ and let $\VV$ be a locally $(<\kappa)$ open refinement of $\OO$ which is guaranteed to exists by (2). For each $x\in X$ choose $V_x\in \VV$ such that $x\in \VV_x$ and an open $W_x\in x$ with $\overline{W_x}\subseteq V_x$, which is
possible to do by the regularity of $X$. Then $\WW=\{W_x:\,x\in X\}$ is a locally $(<\kappa)$ closed refinement of $\OO$ and by Theorem \ref{MichaelTh1}, we shall finish if we can show that $\WW$ is closure preserving. Suppose that $\WW'\subseteq \WW$ and let $x\in\overline{\bigcup \WW'}$. We need to show
that $x\in \bigcup\{\overline{W}:\,W\in\WW'\}$. Let $O$ be an open set containing $x$ which intersects only $<\kappa$ many members of $\WW'$,
say $\{W_\alpha:\,\alpha<\alpha^\ast\}$ for some $\alpha^\ast<\kappa$. If $x\notin \overline{W_\alpha}$ for any $\alpha$, then for each $\alpha<\alpha^\ast$ we can find an open $O_\alpha$
containing $x$ such that $O_\alpha\cap W_\alpha=\emptyset$. But then $O\cap \bigcap_{\alpha<\alpha^\ast}O_\alpha$ is an open set (as $X$ is $(<\kappa)$-open)
which contains $x$ and is disjoint from $\bigcup \WW'$- a contradiction. Hence $x\in\overline{W_\alpha}$ for some $\alpha$ and we are done.
$\eop_{\ref{T:GMK}}$
\end{proof}

We can now state the main lemma that is used to connect the $\nabla$ and the $\square$ products, as one may see in several applications.
For $\kappa=\aleph_0$ this is Kunen's Lemma 1.4 in 
\cite{MR514975}.

\begin{lemma}\label{L:6.8d} Suppose that $f:X\to Y$ is a closed continuous function and:
\begin{itemize}
\item $X$ is regular and $(<\kappa)$-open,
\item  for every $y\in Y$, $f^{-1}(\{y\})$ is $(<\kappa^+)$-compact and
\item $Y$ is paracompact.
\end{itemize}
Then $X$ is paracompact. 
\end{lemma}

\begin{proof} Since paracompactness is inherited by closed subspaces, we have that $f(X)$ is paracompact and therefore it suffices to show the
theorem in the case $f(X)=Y$.
By Lemma \ref{T:GMK}, it suffices to show that an arbitrary open cover of $X$ has a
refinement which can be written as a union of $\kappa$ many locally finite families. Let $\UU$ be an open cover of $X$. For each $y\in Y$, choose a subcollection
$\{U^y_\alpha:\alpha < \kappa\}$ of $\UU$ which covers $f^{-1}\{y\}$. Let
\[
V^y = \left\{z\in Y : f^{-1}\{z\} \subseteq \bigcup_{\alpha < \kappa}U^y_\alpha\right\}\,.
\]
Note that $Y\setminus V^y =f(X\setminus \bigcup_{\alpha < \kappa}U^y_\alpha)$ since $f$ is a surjection. Since $f$ is closed, we have that  $Y\setminus V^y$ is closed and therefore, $V^y$ is open and contains $y$. Therefore, $\{V^y:y\in Y\}$ is an open cover of $Y$. 
Using the paracompactness of $Y$, we can find a locally finite open refinement $\UU'$ of $\{V^y:y\in Y\}$. For each $U\in \UU'$ choose $g(U)\in Y$
such that $U\subseteq V^{g(U)}$. For $y\in Y$ let $W^y=\bigcup\{ U\in \UU':\,g(U)=y\}$. Hence 
$\{W^y:y\in Y\}$ is a locally finite open refinement of $\{V^y:y\in Y\}$. Let 
\[
\OO_\alpha=\{U^y_\alpha \cap f^{-1}(W^y):y\in Y\}
\]
and let $\OO=\bigcup_{\alpha<\kappa} \OO_\alpha$. Clearly, $\OO$ forms an open refinement of $\UU$. It is still a cover of $X$, because if $x\in X$ and $y=f(x)$, then $y\in W^y$ and there is some 
$\alpha<\kappa$ such that $x\in U^y_\alpha$. Therefore $x\in U^y_\alpha \cap f^{-1}(W^y)$. We claim that each $\OO_\alpha$ is locally finite. 

To see that, let $x\in X$ and let $y=f(x)$. We can choose an open set $W$ which contains $y$ and which intersects only finitely many of
$\{W^y:y\in Y\}$. Let $O=f^{-1}(W)$, so $x\in O$ and $O$ is open in $X$. If $x'\in O\cap f^{-1}(W^{y'})$ for some $y'\in Y$, then
$f(x)=y\in W^{y'}\cap W$, hence there can only be finitely many such $y'$. In particular, each $\OO_\alpha$ is locally finite.
$\eop_{\ref{L:6.8d}}$
\end{proof}

\subsection{Proof of Theorem \ref{big}} (a) Let  $X={}^{<\kappa}\bigbox_{\alpha<\kappa} X_\alpha$. Since each $X_\alpha$ is regular and
$(<\kappa)$-open, then so is $X$. The weight of $X$ is $\le\kappa$, since we have assumed that the weight of each $X_\alpha$ is $\le\kappa$ and we
have $\kappa^{<\kappa}=\kappa$. So $X$ is $(<\kappa)^+$-compact. By Lemma \ref{notchtheorem}, $X$ is paracompact.

\smallskip

{\noindent (b)} Let $X=\bigbox_{\alpha<\kappa} X_\alpha$. We first note that since each $X_\alpha$ is regular and
$(<\kappa)$-open, then so is $X$. We plan to use Lemma 
\ref{L:6.8d} applied to the quotient map $q: X \to 
\nabla_{\alpha<\kappa} X_\alpha$.

For $[ x]\in \nabla_{\alpha<\kappa} X_\alpha$, consider 
\[
q^{-1}\left( [ x] \right) = \bigcup_{A\in [\kappa]^{<\kappa}} \left( \bigbox_{\alpha \in A}X_\alpha \times 
\bigbox_{\alpha \notin A}\{x_\beta\} \right)\,.
\]
Since $\kappa^{<\kappa}=\kappa$ and $w(X_\alpha)\leq \kappa$ for every $\alpha < \kappa$, we have that for each $A\in [\kappa]^{<\kappa}$, 
$w(\bigbox_{\alpha \in A}X_\alpha)=\kappa$. Hence, $w\left(q^{-1}\left([x] \right)\right)=\kappa$. Therefore, $q^{-1}\left( [x] \right)$ is $(<\kappa^+)$-compact.

On the other hand, since $w(X)\leq 2^{\kappa}$, the weight of $\nabla_{\alpha<\kappa} X_\alpha$ 
is also $\leq 2^{\kappa}$ and therefore, it is $(<(2^{\kappa})^+)$-compact. Since it is also 
$(<2^{\kappa})$-open by Corollary \ref{C:6.11d}(b), by Lemma \ref{notchtheorem} we conclude  that it is paracompact.

By Lemma \ref{C:6.12d}, $q$ is a closed map. It is clearly a continuous surjection. For each $[x]$, the preimage $q^{-1}\left( [x] \right)$ is $(<\kappa^+)$-compact, 
and $X$ is regular and $(<\kappa)$-open, while $\nabla_{\alpha<\kappa} X_\alpha$  is paracompact.
We conclude by Lemma \ref{L:6.8d} that $\bigbox_{\alpha<\kappa} X_\alpha$ is
paracompact.

\smallskip 

{\noindent (c)} This time let $X={}^{<\kappa^+}\bigbox_{\alpha<\kappa^+} X_\alpha$. As before, $X$ is regular and $(<\kappa)$-open. 
We plan to use Lemma 
\ref{L:6.8d} applied to the quotient map $q: X \to {}^{<\kappa^+}\nabla_{\alpha<\kappa^+} X_\alpha$. In particular, we need to show that each $q^{-1}\left( [x] \right)$ is 
$(<\kappa^+)$-compact. The proof is different than in the case (b) so we give it as a lemma.

\begin{lemma}\label{L:GML5} For each 
$x\in X$, the fiber $q^{-1}([ x])$ is $(<\kappa^+)$-compact.
\end{lemma}

\begin{proof} (of Lemma \ref{L:GML5}) Let $\U$ be an open cover of $q^{-1}([ x])$ by basic open sets of $X$ and we shall find
a subcover $\FF$ of size $\le\kappa$. We start by defining an auxiliary function. Suppose that 
$F\in [\kappa^+]^{<\kappa}$ and consider the product 
$W_F = \prod_{\beta < \kappa^+}W_\beta$, where 
\[
W_\beta =\begin{cases}
 \{x_\beta\} &\mbox{if } \beta\notin F\\
X_\alpha  &\mbox{otherwise.}\\
\end{cases}
\]
Then since $|F|<\kappa$ and each $X_\beta$ is assumed to have weight $\le\kappa$, we have that the weight of $W_F$ is at most 
$\kappa^{<\kappa}=\kappa$ and in particular $W_F$ is $(<\kappa^+)$-compact. In particular, there is a subfamily $\UU_F$ of $\UU$ with
$|\UU_F|\le\kappa$ such that $W_F\subseteq \bigcup \UU_F$. Let us denote by $f$ a function which to each $F\in [\kappa^+]^{<\kappa}$ assigns
$f(F)=\UU_F$.

For each $\alpha<\kappa^+$ let 
\[
D_\alpha=\{y\sim x:\,\{\beta:\,y_\beta\neq x_\beta\}\subseteq \alpha\}.
\]
Let $\chi$ be a regular cardinal large enough for the all objects mentioned so far to be elements of $\HH(\chi)$.
By the assumption that $\kappa^{<\kappa}=\kappa$ we can find $M\elementary \HH(\chi)$ such that
$\kappa\subseteq M$, $|M|=\kappa$, $M={}^{<\kappa}M$, and
such that $\kappa, X,\UU, f, x\in M$. It follows that $M\cap\kappa^+$ is an ordinal $\delta$ of cofinality $\kappa$. Let $\FF=\UU\cap M$.
We can observe that the following properties are satisfied by $\FF$ and $\delta$:
\begin{itemize}
\item $\F \subseteq \U$ is of size $|\F|\leq\kappa$,
\item for every $V\in\F$, we have $\supt(V)\subseteq\delta$, 
\end{itemize}

[Why? Because if $V\in \FF$ then $V\in M$ and so $\supt(V)\in M\cap [\kappa^+]^{<\kappa}=[\delta]^{<\kappa}$.]

\begin{itemize}
\item for every $\alpha<\delta$ and $y\in D_\alpha$, there exists $V\in\F$ such that $y\in V$.
\end{itemize}

[Why? Let $\alpha<\delta$ and $y\in D_\alpha$. Let $F=\{\beta:\;y_\beta\neq x_\beta\}$, so $F\in M\cap [\kappa^+]^{<\kappa}$. In particular, $f(F)\in M$.
Since $M$ is closed under $(<\kappa)$-sequences, it contains as subsets all its elements of size $\le\kappa$ and in particular $f(F)\subseteq M\cap \UU=\FF$.
The conclusion follows since there is $V\in f(F)$ which contains $y$.]

We now claim that $q^{-1}([x])\subseteq \bigcup \FF$, that is $\FF$ is a subcover of $\UU$ of size $\le\kappa$.
Indeed, given any $z\sim x$, let $y=z\upharpoonright \delta\cup 
x\upharpoonright (\kappa^+\setminus \delta)$. Since $\cf(\delta)=\kappa$, there must be $\alpha<\delta$ such that $y\in D_\alpha$.
By the above, there exists $V\in\F$ such that $y\in V$. Since $\supt(V)\subseteq \delta$, we have that $V$ contains any point which agrees with $y$
up to $\delta$, in particular $z$. 
We have that $y\in V$. $\eop_{\ref{L:GML5}}$
\end{proof}

\begin{lemma}\label{L:GML4}The quotient map  
$q: X \to {}^{<\kappa^+}\!\nabla_{\alpha<\kappa^+} X_\alpha$ is a closed map.
\end{lemma}

\begin{proof} Suppose that $K$ is a closed subset of $X$ and that $[ x]\notin q(K)$.

\begin{claim}\label{L:GML3.5} For every $F\in [\kappa^+]^{<\kappa}$ there exists a basic open set 
\[
U_F = {}^{<\kappa^+}\!\bigbox_{\alpha<\kappa^+} U^F_\alpha
\]
such that $x\in U_F$, $U_F\cap K=\emptyset$, 
and $U^F_\alpha = X_\alpha$ for all $\alpha \in F$.
\end{claim}

\begin{proof}(of Claim \ref{L:GML3.5}) By the same proof as that of Claim \ref{cl:closedmap}, we can show that the natural projection $\pr_F : X  \to
{}^{<\kappa^+}\!\bigbox_{\alpha\notin F} X_\alpha$ is a closed map. Let $z = x\upharpoonright (\kappa^+\setminus F)$.

Then $z\notin \pr_F(K)$ because 
$[ x]\notin q(K)$. There exists $V = {}^{<\kappa^+}\!\bigbox_{\alpha\notin F} V_\alpha$ such that $z\in V$ and $V \cap \pr_F(K)=\emptyset$. Let 
\[
U^F_\alpha = 
\begin{cases}
X_\alpha &\mbox{if }\alpha \in F,\\
V_\alpha &\mbox{if }\alpha \notin F
\end{cases}
\]
and let $U_F$ be as in the statement of the claim. Then $U_F\cap K=\emptyset$ and $x\in U_F$, as required.
$\eop_{\ref{L:GML3.5}}$
\end{proof}

Let $M$ be as in the proof of Lemma \ref{L:GML5}, but this time assume that the function $F\mapsto U_F$ for $U_F$ from Lemma \ref{L:GML3.5} is also an element of $M$.
Then for every $F\in [\delta]^{<\kappa}$, the support $\supt(U_F)$ is contained in $\delta$ and is disjoint from $F$.

Let $\langle F_\gamma:\, \gamma< \kappa\rangle$ be an increasing sequence of sets, with $|F_\gamma|< \kappa$ and $\bigcup_{\gamma< \kappa} F_\gamma = \delta$.
For every $\beta < \kappa^+$ let $V_\beta = \bigcap_{\beta\notin F_\gamma} U^{F_\gamma}_\beta$. 
For $\beta < \delta$, this is an intersection of $<\kappa$ many open sets, while for 
$\beta \geq \delta$, $V_\beta = X_\beta$ because the support of $U_{F_\gamma}$ is contained in $\delta$. Note also that if $\beta < \delta$ and $\beta \in F_0$, and so in all of the $F_\gamma$, then 
$V_\beta = X_\beta$ because there is no $\gamma<\kappa$ such that $\beta\notin F_\gamma$. 
Finally, let $V = \prod_{\beta < \kappa^+} V_\beta$. Then clearly $x\in V$.

To conclude the proof, we show that $q(V) \cap q(K) = \emptyset$. Suppose that 
$[ y] = [ z]$ for some $y\in V$ and $z\in K$. Let $\gamma < \kappa$ such that 
$\{\alpha < \delta : y_{\alpha} \neq z_{\alpha}\} \subseteq F_\gamma$. By definition of $V$, we have 
$y\in U_{F_\gamma}$, and by definition of $U_{F_\gamma}$, we have $z\in U_{F_\gamma}$, 
contradicting the fact that $U_{F_\gamma}\cap K = \emptyset$.
$\eop_{\ref{L:GML4}}$
\end{proof}

Next we need to show that ${}^{<\kappa^+}\!\bigbox_{\alpha<\kappa^+} X_\alpha$ is paracompact. By $2^\kappa=\kappa^+$ we have that 
the weight of 
${}^{<\kappa^+}\!\bigbox_{\alpha<\kappa^+} X_\alpha$ is $\kappa^+$, therefore so is the weight of
${}^{<\kappa^+}\!\nabla_{\alpha<\kappa^+} X_\alpha$. 

Since it is $\kappa^+$-open by 
Lemma \ref{lm:openess}, we have that ${}^{<\kappa^+}\!\nabla_{\alpha<\kappa^+} X_\alpha$ is paracompact by
Lemma \ref{notchtheorem}. Finally, we are in the position to apply Lemma 
\ref{L:6.8d}.
$\eop_{\ref{big}}$

\section{A result relating to $\kappa$-scales}\label{sec:scaling} Combining the methodology of $\kappa$-metrisability and $(<\kappa)$-openness from previous sections, we obtain the following Theorem \ref{T:6.14dd}. Its instance for $\kappa=\aleph_0$ is due to van Douwen as Theorem 10.6 in \cite{MR564100} and the proof here is really just a generalisation of his techniques.

\begin{thm}\label{T:6.14dd} Suppose that $\kappa^{<\kappa} = \kappa$ and that there exists a scale in $(\kappa^\kappa,\le^\ast_\kappa)$. 

Then, if for each $\alpha<\kappa$, $X_\alpha$ is a $\kappa$-metrisable $\kappa$-Sikorski space,
the product $X=\bigbox_{\alpha<\kappa} X_\alpha$ is paracompact.
\end{thm}

This section is devoted to the proof of Theorem \ref{T:6.14dd}. Note that 
a scale may not always exist in $(\kappa^\kappa,\le^\ast_\kappa)$, but if it does, its length is a uniquely determined regular cardinal.
Let us fix the spaces 
as in the statement and a scale $\bar{f}=\langle f_i:\,i<\lambda\rangle$ in $(\kappa^\kappa,\le^\ast_\kappa)$,
so $\lambda$ is a regular cardinal and by the definition of a scale, it follows that $\lambda\in[\kappa^+, 2^\kappa]$. 
\begin{lemma}\label{L:6.16d} The Nabla product $\nabla_{\alpha<\kappa} X_\alpha$ is $\lambda$-metrisable.
\end{lemma}

\begin{proof}(of Lemma \ref{L:6.16d}) For each $\alpha<\kappa$ define a function
\[
U_\alpha: X_\alpha \times \kappa \to \tau(X_\alpha)
\]
that witnesses the $\kappa$-metrisability of $X_\alpha$ (see Definition \ref{def:kappametrisable}). Let $\langle f_i:\,i<\lambda\rangle$ be a scale
in  ${}^{\kappa}{\kappa}$.

Let $x=(x_\alpha)\in X$ and $f\in {}^{\kappa}{\kappa}$.  
Define an open set $V(x,f)$ in $\nabla_{\alpha<\kappa} X_\alpha$ containing $[x]$ by
\[
V(x,f) = \nabla_{\alpha<\kappa} U_\alpha(x_\alpha,f(\alpha)) \,.
\]
It is clear that $V(x,f)$ is indeed an open set and that $\{V(x,f) : f\in {}^{\kappa}{\kappa}\}$ is a neighbourhood base at 
$[x]\in\nabla_{\alpha<\kappa} X_\alpha$.
[Why? Because every open neighbourhood of $x\in X$ contains one of the form $\prod_{\alpha<\kappa} U_\alpha(x_\alpha,f(\alpha))$, 
for some $f\in{}^\kappa\kappa$. ]

We now observe that
for all $x,y \in \bigbox_{\alpha<\kappa} X_\alpha$  and $f\leq^\star_\kappa g\in {}^{\kappa}{\kappa}$:
\begin{enumerate}[(a)]
\item if $[y] \in V(x,f)$ then $V(y,g) \subseteq V(x,f)$, 
\end{enumerate}
[Why ? Let $\alpha_0 < \kappa$ such that
$f(\alpha) \leq g(\alpha)$ for all $\alpha \geq \alpha_0$. Since $[y] \in V(x,f)$, there exists $\alpha_1\ge\alpha_0$ such that for all 
$\alpha \geq \alpha_1$, $y_\alpha \in U_\alpha(x_\alpha,f(\alpha))$ and hence $y\in U_\alpha(y_\alpha, g(\alpha))\subseteq U_\alpha(x_\alpha,f(\alpha))$
for $\alpha\ge\alpha_1$. This implies that $V(y,g) \subseteq V(x,f)$.]

\begin{enumerate}[(b)]
\item if $[y]\notin V(x,f)$ then $V(y,g) \cap V(x,f) = \emptyset$.
\end{enumerate}
[Why? Since $[y] \notin V(x,f)$, we have $\sup(\{\alpha<\kappa:\,y_\alpha \notin U_\alpha(x_\alpha,f(\alpha))\})=\kappa$. Therefore
$U_\alpha(y_\alpha,g(\alpha)) \cap U_\alpha(x_\alpha,f(\alpha)) = \emptyset$ for cofinally many $\alpha$. In particular, $V(y,g) \cap V(x,f) = \emptyset$.]

It follows that a function $U$ that witnesses that $\nabla_{\alpha<\kappa} X_\alpha$ is 
$\lambda$-metrisable is given by letting $U(x,i) = V(x,f_i)$ all $i < \lambda$.
$\eop_{\ref{L:6.16d}}$
\end{proof}

The proof now follows similarly as in the proof of Theorem \ref{big}, using Lemma \ref{L:6.8d}. The assumptions imply that $X$ is regular and $(<\kappa)$-open. Taking the quotient map $q$ for $f$ and the nabla $\nabla_{\alpha<\kappa} X_\alpha$ as $Y$, we have that $q$ is continuous by definition and closed by Lemma \ref{C:6.12d}. By the remark after Definition \ref{def:kappametrisable}, $Y$ is paracompact since in Lemma \ref{L:6.16d} we have shown that it is
$\lambda$-metrisable. Finally, being $\kappa$-metrisable in conjunction with $\kappa^{<\kappa}=\kappa$ implies that each $X_\alpha$ has weight
$\le\kappa$. Hence Lemma \ref{L:GML5} implies that for each 
$x\in X$, the fiber $q^{-1}([ x])$ is $(<\kappa^+)$-compact. The conclusion follows by Lemma \ref{L:6.8d}.
$\eop_{\ref{T:6.14dd}}$

\section{Products related to singular cardinals}\label{sec:aleph-omega} The first theorem in this section can be proved using a technique discovered by 
M.E. Rudin in her celebrated construction of a Dowker space announced in \cite{MR270328} and written in \cite{RudinDowker}. This construction has been surveyed in various references, including a very clearly written one by Klaas Pieter Hart in \cite{kphart}. 

\begin{theorem}\label{th:aleph_omega} The box product $\bigbox_{n<\omega}(\omega_n+1)$ is paracompact, in fact ultraparacompact.
\end{theorem}

We remind the reader that the spaces $\omega_n+1$ are considered in the topology where $\omega_n$ is discrete and the point $\omega_n$ has 
a neighbourhood base consisting of co-bounded subsets of $\omega_n$. This is a different topology from the one considered by Rudin.

\smallskip

\begin{proof} We have that $\bigbox_{n<\omega} (\omega_n+1) = (\omega+1) \times \bigbox_{1\le n<\omega} (\omega_n+1)$. Since $\omega+1$ is compact
and 0-dimensional, if we show that $X=\bigbox_{1\le n<\omega}(\omega_n+1)$ is ultraparacompact, we have that 
$\bigbox_{n<\omega} (\omega_n+1)$ is ultraparacompact by Lemma \ref{lem:compxpara}.

As a box product of $(<\omega_1)$-open spaces, $X$ is $(<\omega_1)$-open. We consider
$X$ as an ordered space with the coordinatewise order, namely $\langle x_n:\,1\le n<\omega\rangle \le \langle y_n:\,1\le n<\omega\rangle$ if for all $n$,
$x_n\le y_n$. The notation $[x,y]$ is used for the intervals in this order.

Let $\OO$ be an open cover of $X$. We shall construct a sequence $\bar{\U}=\la\U_\alpha : \alpha < \omega_1\ra$ of clopen covers 
of $X$ such that each $U\in\U_\alpha$ is of the form $\prod_{1\le n<\omega} U_n$, where $U_n=\{x_n\}$ or $U_n = [x_n,\omega_n]$, 
for some $x_n<\omega_n$. In particular, $U$ is of the form $[x,y]$ for some $x\le y$ in $X$. Moreover:
\begin{enumerate}[(1)]
\item each $\U_\alpha$ is a pairwise disjoint clopen cover of $X$,
\item if $\alpha < \beta$ then $\U_\beta$ is a refinement of $\U_\alpha$,
\item if $U\in \U_\alpha$ and $U\subseteq O$ for some $O\in\OO$ then $U\in\U_{\alpha +1}$ and
\item if $U\in \U_\alpha$ is of the form $[x,y]$, and there is no $O\in \OO$ such that $U\subseteq O$, then for every 
$V\in\U_{\alpha +1}$ with $V\subseteq U$ and $V = [u,v]$, there is some $n$ such that $u_n = v_n < y_n=\omega_n$ or 
$V \subseteq O$ for some $O\in \OO$.
\end{enumerate}

Supposing that $\bar{\U}$ has been constructed,  let $y \in X$ and for $\alpha < \omega_1$ denote the unique element of $\U_\alpha$ that contains $y$ by 
$[u_\alpha, v_\alpha]$. Note that for all $n$, either $(u_\alpha)_n = (v_\alpha)_n = y_n$, or $(v_\alpha)_n = \omega_n$ and 
$y_n \in [(u_\alpha)_n, \omega_n]$. Letting $\alpha$ increase, we either have that $[u_\alpha, v_\alpha]\subseteq O$ for some $O\in\OO$, in which case for every $\beta\ge\alpha$ we have that $[u_\beta, v_\beta]=[u_\alpha, v_\alpha]$, or otherwise $[u_\beta, v_\beta]\subseteq [u_\alpha, v_\alpha]$ and
there is some $n<\omega$ such that $(u_\beta)_n=(v_\beta)_n<(v_\alpha)_n$.
By a cofinality argument, it then follows that:
\begin{enumerate}[(a)]
\item for every $1\le n<\omega$ there is an $\alpha_n$ such that $(v_\alpha)_n = (v_{\alpha_n})_n$ for all $\alpha \geq \alpha_n$. 
Let $\alpha_y = \sup_n \alpha_n$ and $\beta_y = \alpha_y +1$,
\item there is an $O \in \OO$ with $[u_{\beta_y}, v_{\beta_y}] \subseteq O$,
\item if $\beta\geq\beta_y$ then $[u_\beta, v_\beta]  = [u_{\beta_y}, v_{\beta_y}]$.
\end{enumerate}

Then it follows from (1) - (4) that the family  $\{[u_{\beta_y}, v_{\beta_y}] : y \in X\}$ is a pairwise disjoint open 
refinement of $\OO$. This shows that $X$ is ultraparacompact, provided we can construct $\bar{\U}$.
Let us show how to do it. 

We start with $\U_0 = \{X\}$. Note that $X = [0, t]$, 
where $t_n = \omega_n$ for all $1\le n\in\omega$.

To construct $\U_{\alpha +1}$ from $\U_\alpha$ let $U \in \U_\alpha$, say $U = [x,y]$. If there is an $O \in \OO$ with 
$U \subseteq O$ put $\I_U = \{U\}$. Otherwise, if there is no such $O$, it cannot be the case that $(x)_n=(y)_n$ for all $n$,
since $\OO$ is a cover of $X$. Also, $y$ belongs to some element of $\OO$. By intersecting such a set with 
$[x,y]$ if necessary, we can find a $z<y$ such that $x<z$ and $[z,y] \subseteq O$ for some $O\in\OO$. (Note that $x<z$ means that $x\le z$ but not
$x=z$, so simply, there exists $n$ such that $(x)_n<(z)_n$.)

For $A\subseteq \omega\setminus\{0\}$ let $\VV_A$ consists of all, if any, sets of the form $\prod_{1\le n<\omega} W_n$ which satisfy
\begin{itemize}
\item $n\in A\implies W_n = \{w_n\}\text{ for some }x_n\leq w_n< z_n$ and
\item $n\notin A\implies W_n = [z_n,y_n]$.
\end{itemize}
By the observation above, $\VV_A$ is defined for at least some non-empty $A$.
Set $\I_U = \bigcup\{\VV_A : \,A\subseteq \omega\setminus \{0\} \}$. Note that $U = \bigcup\I_U$ and $\I_U$ is a pairwise disjoint family of clopen sets.
Finally, let $\U_{\alpha +1} = \bigcup\{\I_U: U\in \U_\alpha\}$.

For a limit ordinal $\alpha<\omega_1$ we proceed in the following way: for $y \in X$ and $\beta<\alpha$, denote the unique element of $\U_\beta$ that contains $y$ by $[u_\beta(y), v_\beta(y)]$. Let $[u_\alpha(y), v_\alpha(y)] = 
\bigcap_{\beta<\alpha} [u_\beta(y), v_\beta(y)]$ which is clopen since $X$ is $(<\omega_1)$-open.
Note that $(u_\alpha(y))_n = (v_\alpha(y))_n = y_n$, or $(v_\alpha(y))_n = \omega_n$ and 
$y_n \in [(u_\alpha(y))_n, \omega_n]$. Note also that for $x\neq y$, either $[u_\alpha(x), v_\alpha(x)] = [u_\alpha(y), v_\alpha(y)]$ or $[u_\alpha(x), v_\alpha(x)] \cap [u_\alpha(y), v_\alpha(y)] = \emptyset$.
Finally, let $\U_{\alpha} = \bigcup\{[u_\alpha(x), v_\alpha(x)]: x\in X\}$.

It is easy to verify that the construction of $\bar{\U}$ verifies all the requirements and hence $X$ has been shown to be ultraparacompact.
$\eop_{\ref{th:aleph_omega}}$
\end{proof}

We note the main points of the proof of Theorem \ref{th:aleph_omega} in order to be able to generalise it. It was important that for $n\ge 1$ the space
was $(<\omega_1)$-open and that $\omega=\cf(\omega)<\omega_1$. So, with minor modifications, the same proof gives the following.

\begin{theorem}\label{th:gen-aleph_omega} Suppose that $\langle \lambda_i:\,i<\kappa\rangle$ is a sequence of ordinals of which all
but finitely many are uncountable and satisfy $\cf(\lambda_i)>\kappa$. 

Then $\bigbox_{i<\kappa}(\lambda_i+1)$ is ultraparacompact.
\end{theorem}

In particular, we obtain as a corollary Theorem \ref{MERudinth}, by letting $\kappa=\aleph_0$ and each $\lambda_i=\omega_1$, for $i<\omega$.

\section{Large cardinals}\label{sec:large_car} We have mentioned the known fact that $\kappa$ is a strongly compact cardinal iff for every $\lambda$, the space ${}^{<\kappa}\square \,2^\lambda$ is $(<\kappa)$-compact. Let us also comment on weakly compact cardinals. For an uncountable $\kappa$ satisfying 
$\kappa^{<\kappa}=\kappa$, we have that 
${}^{<\kappa}\square \,2^\kappa$ is $(<\kappa)$-compact iff $\kappa$ is a weakly compact cardinal (see Theorem 5.1 in \cite{BuhagiarDzamonja1}). The case when $\kappa$ is singular does not seem to have been treated in the topology we consider here. Specifically, Luca Motto Ros proves in Theorem 5.6 in \cite{luca} that for $\kappa$ singular, the space ${}^{<<\kappa}\square\, 2^\kappa$ (see Definition \ref{boundedtopology} below) is not $(<\kappa)$-compact, but the topology in question is
the bounded topology, that is only the boxes of bounded sizes are taken into the base. We give a proof of the analogous fact for the topology ${}^{<\kappa}\square \,2^\kappa$. This adds to  the general theory of weakly compact cardinals and we do obtain a bit more by exhibiting a cover of
${}^{<\kappa}\square\, 2^\kappa$ consisting of $\ge\kappa$ pairwise disjoint basic open sets.

\begin{theorem}\label{singularnotcompact} Suppose that $\kappa^{<\kappa}>\kappa$ (which is in particular true if $\kappa$ is singular). Then the space ${}^{<\kappa}\square\, 2^\kappa$ is not $(<\kappa)$-compact and moreover, it has a cover of size $\ge\kappa$ consisting of pairwise disjoint basic open sets.
\end{theorem}

\begin{proof} In fact, we construct a covering of $X={}^{<\kappa}\square\, 2^\kappa$ consisting of disjoint basic open sets and which has size $\ge\kappa$.

{\noindent (1)} First suppose that $\kappa$ is singular.
Let $\theta=\cf(\kappa)$ and let $\langle \kappa_i:\, i<\theta\rangle$ be an increasing continuous sequence of cardinals cofinal in $\kappa$ and such that 
$\kappa_0=\theta$. Let us partition $2^\theta=\bigcup_{i<\theta} A_i$ where each $A_i$ has size $2^\theta$.

For each $f\in 2^\kappa$ define a basic clopen set $U_f$ as follows. First, let $i^\ast=i^\ast(f)<\theta$ be such that $f\rest\theta\in A_{i^\ast}$. Then 
for any $\alpha<\kappa$,
the projection $\pi_\alpha(U_f)$ is defined as follows:
\[
\pi_\alpha(U_f)=
\begin{cases}
\{ f(\alpha)\} & \mbox{if } \alpha<\theta\mbox{ or } i\ge i^\ast \mbox{ and }\alpha\in [\kappa_i,\kappa_i+ \kappa_{i^\ast}), \\
\{0,1\} & \mbox{otherwise.}
\end{cases}
\]
Since the support of $U_f$ has size $\le \theta\cdot \kappa_{i^\ast} =  \kappa_{i^\ast} < \kappa$, the set $U_f$ is basic clopen in ${}^{<\kappa}\square\, 2^\kappa$.
Clearly, $f\in U_f$, hence $\mathcal U=\{U_f:\,f\in 2^\kappa\}$ covers $X$. Note that $\mathcal U=\bigcup_{i<\theta}\{U_f:\, i^\ast(f)=i\}$. 
Let $\UU_i=\{U_f:\, i^\ast(f)=i\}$.

Suppose that $f,g$ are such that $i^\ast(f)=i^\ast(g)=i^\ast$. Then either for all $i\ge i^\ast$ we have $f\rest [\kappa_i,\kappa_i+ \kappa_{i^\ast})=
g\rest [\kappa_i,\kappa_i+ \kappa_{i^\ast})$, in which case $U_f=U_g$, or not, in which case $U_f\cap U_g=\emptyset$.
Hence each $\UU_i$ consists of exactly $2^{\kappa_i}$ pairwise disjoint sets. On the other hand, if $f$ and $g$ are such that $i^\ast(f)\neq i^\ast (g)$,
then in particular $f\rest \theta\neq g \rest \theta$ and hence $U_f\cap U_g=\emptyset$.

In conclusion, the cover $\UU$ consists of $\sup_{i<\theta} 2^{\kappa_i}=2^{<\kappa}\ge \kappa$ pairwise disjoint basic open sets.

{\noindent (2)} Now suppose that $\kappa$ is regular and $\kappa^{<\kappa}>\kappa$. Let $\theta<\kappa$ be such that $2^\theta>\kappa$.
For $A\subseteq \theta$ let $O_A$ be the basic open set in $X$ consisting of points whose restriction to $\theta$ is $\chi_A$. Let
$\OO=\{O_A:\,A\subseteq \theta\}$. Then $\OO$ is an open cover of $X$ consisting of $>\kappa$ pairwise disjoint open sets.
$\eop_{\ref{singularnotcompact}}$
\end{proof}

\begin{corollary}\label{weakcom} An uncountable cardinal $\kappa$ is weakly compact iff ${}^{<\kappa}\square \,2^\kappa$ is $(<\kappa)$-compact.
\end{corollary}

Consider now the question of paracompactness of the products of the form ${}^{<\kappa}\square \,2^\lambda$. It seems that requiring some level of paracompactness of the products of the type ${}^{<\kappa}\square \,2^\lambda$  for $\lambda\ge\kappa$ should lead to a large cardinal property

\begin{definition}\label{} A cardinal $\kappa$ is {\em $\lambda$-paracompact} for a cardinal $\lambda\ge\kappa$ if the space ${}^{<\kappa}\square \,2^\lambda$ is paracompact. We say that
$\kappa$ is {\em paracompact} if it is $\lambda$-paracompact for every $\lambda\ge\kappa$.
 \end{definition}

For $\kappa=\lambda$ we obtain from Theorem \ref{big}(a) that $\kappa^{<\kappa}=\kappa$ suffices for 
${}^{<\kappa}\square \,2^\kappa$ to be paracompact. In fact just $\kappa$ being regular is enough, as shown by the following Theorem \ref{regular}.
In that theorem we shall use an auxiliary notion which will also be helpful in some further arguments.

\begin{definition}\label{b-paracompact} A topological space is {\em b-ultraparacompact} if every open cover has an open refinement consisting of disjoint basic clopen sets.
\end{definition}

Clearly, being b-ultraparacompact implies being ultraparacompact and hence paracomopact. The definition of
b-paracompactness tacitly assumes that we are working with some fixed base. In the case that will interest us the most, which are products of the
form ${}^{<\kappa}\square \,2^\lambda$, the base we have in mind is the canonical base consisting of sets of the form $[s]=\{f\in 2^\lambda:
\,s\subseteq f\}$, where $s$ is a partial function from $\lambda$ to $2$ whose domain has size $<\kappa$.

\begin{theorem}\label{regular} Suppose that $\kappa$ is a regular cardinal. Then ${}^{<\kappa}\square \,2^\kappa$ is paracompact, moreover it is 
b-ultraparacompact.
\end{theorem}

\begin{proof} For each $\alpha<\kappa$ we define
\[
\UU_\alpha=\big\{(\prod_{\beta<\alpha} \{f (\beta)\}\times \prod_{\alpha\le\beta<\kappa} \{0,1\}):\;f\in {}^\alpha 2\big\}.
\]
Hence each $\UU_\alpha$ consists of pairwise disjoint basic clopen sets in $X={}^{<\kappa}\square \,2^\kappa$ and it is a cover of $X$.
Note that if $\beta<\alpha<\kappa$, $U\in \VV_\beta$ and $V\in \VV_\alpha$, then either we have $U\cap V=\emptyset$, or $V \subseteq U$. 

Let $\OO$ be any open cover of $X$ and 
for $\alpha<\kappa$ let 
\[
\VV_\alpha=\{ U\in \UU_\alpha:\,(\exists O\in \OO)\,U\subseteq O\}.
\]
Define for $\alpha<\kappa$
\[
\WW_\alpha=\{ V\in \VV_\alpha:\,(\forall U\in  \bigcup_{\beta<\alpha}\VV_\beta)\,V \nsubseteq U\}=\{ V\in \VV_\alpha:\,(\forall U\in  \bigcup_{\beta<\alpha}\VV_\beta)\,V \cap U=\emptyset\}.
\]
Let $\WW= \bigcup_{\alpha<\kappa} \WW_\alpha$, so $\WW$ is an open refinement of $\OO$ and it consists of disjoint basic clopen sets. We claim that
$\WW$ is a cover of $X$. Given $x\in X$, let $O\in \OO$ be such that $x\in O$. By the definition of the topology on $X$, there exists $\alpha<\kappa$
such that $V=\prod_{\beta<\alpha} \{x (\beta)\}\times \prod_{\alpha\le\beta<\kappa} \{0,1\}\subseteq O$. Hence $V \in \VV_\alpha$ contains $x$.
By the definition of $\WW_\alpha$, either $V\in \WW_\alpha$, or there exists $U\in \bigcup_{\beta<\alpha}\VV_\beta$ such that $V\subseteq U$. Let
$\beta$ be the least $\beta<\alpha $ such that there is $U\in \VV_\beta$ with $V\subseteq U$. Hence $U\in \WW_\beta$. In any case, we have found
an element of $\WW$ which contains $x$.
$\eop_{\ref{regular}}$
\end{proof}

We conclude that there is no large cardinal or even cardinal arithmetic property involved in the paracompactness of ${}^{<\kappa}\square \,2^\kappa$ for regular $\kappa$.
Since being $(<\kappa)$-compact implies being $(<\kappa^+)$-compact and every space of the form ${}^{<\kappa}\square \,2^\lambda$ for 
$\lambda\ge\kappa$ is regular and $(<\kappa)$-open whenever $\kappa$ is a regular cardinal, we obtain by Lemma \ref{notchtheorem} that every strongly compact cardinal is paracompact. In fact, we shall show that it is b-ultraparacompact.

\begin{theorem}\label{stronglycompactimpliesbparacompact} Suppose that $\kappa$ is a strongly compact cardinal. Then for every $\lambda\ge\kappa$, the space ${}^{<\kappa}\square \,2^\lambda$ is b-ultraparacompact.
\end{theorem}

\begin{proof} Let $\kappa$ be a strongly compact cardinal, $\lambda\ge\kappa$ and let $\UU$ be an open cover of $X={}^{<\kappa}\square \,2^\lambda$.
We may refine it so that it without loss of generality consists of basic clopen sets. Since $\kappa$ is strongly compact, $X$ is $(<\kappa)$-compact and hence we can assume by passing to a subcover if necessary, that $|\UU|<\kappa$. We can enumerate $\UU=\{U_\alpha:\,\alpha<\mu\}$ for
some $\mu<\kappa$ and for each $\alpha<\mu$ let $V_\alpha\deq U_\alpha\setminus \bigcup_{\beta<\alpha} U_\beta$, which are clopen since $X$ is
$(<\kappa)$-open.

Clearly $\alpha\neq\beta\implies V_\alpha\cap V_\beta=\emptyset$ and all we need to show is that each $V_\alpha$ can be written as the union
of disjoint basic clopen sets. Let us denote by $s_\alpha$ the partial function such that $U_\alpha=[s_\alpha]$, for each $\alpha$. 
Then:
\begin{equation*}
U_\alpha\setminus \bigcup_{\beta<\alpha} U_\beta=
\bigcup \{[f]: \,f\,:\bigcup_{\beta\le\alpha} \dom(s_\beta)\to 2\mbox{ satisfies }(i) \mbox{ and } (ii) \mbox{ below}\}:
\end{equation*}
\begin{description}
\item{\em (i)} $(\forall\beta<\alpha)(\exists i\in \dom(s_\beta)\,f(i)\neq
s_\beta(i)$,
\item{\em (ii)} $(\forall i \in \dom (s_\alpha))f(i)=s_\alpha(i)$.
\end{description}

which is the union of a family of pairwise disjoint basic clopen sets.
$\eop_{\ref{stronglycompactimpliesbparacompact}}$
\end{proof}


To obtain a lower bound for the paracompactness of a cardinal, we shall first show that no (infinite) successor cardinal can be paracompact. The following
lemma will be useful.

\begin{lemma}\label{F:1} The ($<\kappa$)-box product is associative iff $\kappa$ 
is a regular cardinal, by which we mean that for any cardinal $\lambda$, any spaces $\langle X_\alpha:\,\alpha<\lambda\rangle$ and
a partition $\lambda=\bigcup_{\gamma\in\Gamma} I_\gamma$, we have
\[
{}^{<\kappa}\bigbox_{\alpha<\lambda} X_\alpha \text{ is homeomorphic to } {}^{<\kappa}\bigbox_{\gamma\in\Gamma}\left(
{}^{<\kappa}\bigbox_{\alpha\in I_\gamma} X_\alpha\right).
\]
\end{lemma}

\begin{proof} Let $p: \prod_{\alpha<\lambda} X_\alpha \to 
\prod_{\gamma\in\Gamma}\left(\prod_{\alpha\in I_\gamma} X_\alpha\right)$ be the natural bijection, we shall show that it is a homeomorphism. Suppose that
$U = \prod_{\alpha<\lambda}U_\alpha$ is a basic open set in ${}^{<\kappa}\bigbox_{\alpha<\lambda} X_\alpha$. Then clearly for every $\gamma\in\Gamma$
we have that $|\supt(U)\cap I_\gamma|<\kappa$, so we can conclude that $p(U)$ is a basic open set in
${}^{<\kappa}\bigbox_{\gamma\in\Gamma}\left(
{}^{<\kappa}\bigbox_{\alpha\in I_\gamma} X_\alpha\right)$. Hence $p$ is an open function and we still need to show that it is continuous. It is in this part of the
proof that we shall use the regularity of $\kappa$. 

Suppose that $U = \prod_{\gamma\in\Gamma}
\left(\prod_{\alpha\in I_\gamma} U_\alpha\right)$ is a basic open set in the space ${}^{<\kappa}\bigbox_{\gamma\in\Gamma}\left(
{}^{<\kappa}\bigbox_{\alpha\in I_\gamma} X_\alpha\right)$.
Let $U^\gamma = \prod_{\alpha\in I_\gamma} U_\alpha$, so $\supt(U^\gamma)|<\kappa$ for each $\gamma\in\Gamma$. Notice also that
$|\{\gamma\in \Gamma:\,\supt(U_\gamma)\neq\emptyset\}|<\kappa$, so by the regularity of $\kappa$,
we have that 
$|\supt(p^{-1}(U))|<\kappa$, and so $p^{-1}(U)$ is a basic open set in ${}^{<\kappa}\bigbox_{\alpha<\lambda} X_\alpha$.

Note that when $\kappa$ is singular, it is easy to construct examples where
$|\supt(U^\gamma)|<\kappa$ for each $\gamma\in\Gamma$ and $|\{\gamma\in\Gamma:\,\supt(U^\gamma)\neq\emptyset\}|<\kappa$ but $|\supt(f^{-1}(U))| = \kappa$.
$\eop_{\ref{F:1}}$
\end{proof}

\begin{theorem}\label{nosucessor} No infinite successor cardinal is paracompact. In fact, there is no infinite cardinal $\kappa$ such that the space 
${}^{<\kappa^+}\!\square 2^{\kappa^{++}}$ is normal.
\end{theorem}

\begin{proof} The proof is a generalisation of a proof of van Douwen in \cite{VANDOUWEN197771} , whose Theorem B proves the latter statement in the case $\kappa=\aleph_0$. As van Douwen does not give many details, we give a detailed proof.

Borges proves in Theorem 2 of \cite{MR253274} that if $\theta$ is an inifinite regular cardinal, then the space
${}^{<\theta}\bigbox\, \theta^\lambda$ is not normal,  for any $\lambda>\theta$ (here $\theta$ is given the discrete topology). In particular, 
${}^{<\kappa^+}\!\bigbox (\kappa^+)^{\kappa^{++}}$ is not normal, and hence not paracompact. We shall proceed to observe that ${}^{<\kappa^+}\!\bigbox2^{\kappa^{++}}$
and ${}^{<\kappa^+}\!\bigbox (\kappa^+)^{\kappa^{++}}$ are homeomorphic, and then the conclusion will follow.

Let us partition $\kappa^{++} = \bigcup_{\gamma<\kappa^{++}}I_\gamma$, where each $|I_\gamma|=\kappa$ and the sets
$\{I_\gamma:\,\gamma<\kappa^{++}\}$ are pairwise disjoint. Then, by Lemma \ref{F:1}, ${}^{<\kappa^+}\!\bigbox 2^{\kappa^{++}}$ is homeomorphic to 
${}^{<{\kappa^+}}\bigbox_{\gamma<\kappa^{++}}\left({}^{<\kappa^+}\bigbox_{\alpha\in I_\gamma} 2\right)$ and 
 ${}^{<\kappa^+}\bigbox (\kappa^+)^{\kappa^{++}}$ is homeomorphic to ${}^{<\kappa^+}\bigbox_{\gamma<\kappa^{++}}\left({}^{<\kappa^+}\bigbox_{\alpha\in I_\gamma} \kappa^+\right)$. Now note that each ${}^{<\kappa^+}\bigbox_{\alpha\in I_\gamma} 2$ is simply a discrete space of size $2^\kappa$,
and so is ${}^{<\kappa^+}\bigbox_{\alpha\in I_\gamma} \kappa^+$, so they are all homeomorphic. Hence ${}^{<\kappa^+}\!\bigbox2^{\kappa^{++}}$
and ${}^{<\kappa^+}\!\bigbox (\kappa^+)^{\kappa^{++}}$ are homeomorphic.
$\eop_{\ref{nosucessor}}$
\end{proof}

We do not know if ${}^{<\kappa} \,\bigbox 2^\lambda$ can be paracompact for any $\lambda>\kappa$. In addition, what we know leaves open the case of singular cardinals.
We would quite like to prove the following Conjecture \ref{nosingular}, which would then, together with the results earlier in this section, demonstrate that 
paracompactness is indeed a large cardinal notion. 

\begin{conjecture}\label{nosingular} No singular cardinal $\kappa$ is paracompact. Moreover, not even the space ${}^{<\kappa}\square\, 2^\kappa$ is paracompact.
\end{conjecture}

We are not able to prove Conjecture \ref{nosingular}  at the moment, so we offer some partial results instead and connect them with results involving the so called bounded topology (which seems to be more often used in the literature when dealing with at singular cardinals).

\begin{theorem}\label{character} Suppose that $\kappa$ is a singular cardinal. Then for every $x\in X= {}^{<\kappa}\bigbox 2^{\kappa}$ we
have $\chi(x,X) >\cf(\kappa)$.
\end{theorem}

\begin{proof} Let $\theta=\cf(\kappa)$ and let $\langle \kappa_i:\,i<\theta\rangle$ be an increasing cofinal sequence in $\kappa$.
Suppose that $x\in X$ and $\{U_i:\,i<\theta\}$ are basic clopen sets containing $x$. By induction on $i<\theta$ choose
an increasing sequence $\langle \alpha_i:\,i<\theta\rangle$ such that $\alpha_i\notin \supt(U_i)$, which is possible since $i<\theta$ and $|\supt(U_i)|<\kappa$, for each $i$. Then let $O$ be the basic open set $[x\rest \{\alpha_i:\,i<\theta\}]$. This set clearly contains $x$ but it does not contain 
any $U_i$ as a subset, and hence the arbitrarily chosen sets $\{U_i:\,i<\theta\}$ do not form a neighbourhood base at $x$. Therefore
$\chi(x,X)>\theta$.
$\eop_{\ref{character}}$
\end{proof}

\begin{corollary}\label{character-notmetrisable} For a singular $\kappa$, the space ${}^{<\kappa}\square 2^{\kappa}$ is not $\cf(\kappa)$-metrisable.
\end{corollary}

The corollary follows straight from the Definition \ref{def:kappametrisable} of $\cf(\kappa)$-metrisability. It is to be contrasted with the following Theorem 
\ref{boundedmetrisable}.

\subsection{Bounded topology} When $\kappa$ is a regular cardinal, then the topology of ${}^{<\kappa}\square\, 2^\kappa$  is exactly the one
generated by boxes whose support is bounded in $\kappa$. This is not the case for singular cardinals and it turns out that the bounded
topology in such a case has some quite nice properties.

\begin{definition}\label{boundedtopology} Suppose that $\kappa$ is an infinite cardinal. The {\em bounded box topology} on $2^\kappa$ is obtained by
taking as basic clopen sets the sets of the form $[s]$ where $\dom(s)$ is a bounded subset of $\kappa$. We denote the resulting space by
${}^{<<\kappa}\square\, 2^\kappa$.
\end{definition}

\begin{theorem}\label{boundedmetrisable} Suppose that $\kappa$ is an infinite cardinal. Then $X={}^{<<\kappa}\square\, 2^\kappa$ is $\cf(\kappa)$-metrisable and hence
b-ultraparacompact.
\end{theorem}

\begin{proof} In the case that $\kappa$ is regular, the conclusion is given by Theorem \ref{regular}, since the $\mathcal U_\alpha$'s witness the 
$\kappa$-metrisability of $X$.

Suppose that $\kappa$ is singular and let $\theta=\cf(\kappa)$. We shall show that $X$ is $\cf(\kappa)$-metrisable. We need to produce a function $U$ from $X\times\theta$ to the open sets of $X$ which exemplifies Definition \ref{def:kappametrisable}. Let $\langle\kappa_i:\;i<\theta\rangle$ be an increasing sequence of cardinals which is cofinal in $\kappa$.
For each $f\in 2^\kappa$ and $i<\theta$ define 
\[
U(f,i)\deq {f\rest\kappa_i}\times\prod_{\kappa_i\le\alpha <\kappa} \{0,1\}.
\]
Then we have that $\{U(f,i):\,i<\theta\}$ is a neighbourhood base at $f$, because if $[s]$ is a basic open set in $X$ containing $f$, there exists
$i<\theta$ such $\dom(s)\subseteq \kappa_i$. Hence $f\in U(f,i)\subseteq [s]$. If $g\in U(f,i)$ and $i\le j$, then clearly $U(g,j)\subseteq U(f,i)$. Finally, 
if $h\notin  U(f,i)$, then $h\rest\kappa_i \neq g\rest\kappa_i$ and hence $U(g,j)\subseteq U(f,i)$ for any $j\ge i$.

Since the sets $U(f,i)$ are basic clopen, Claim \ref{vanDouwenCl2} gives us not only that $X$ is ultraparacompact, but that it is b-ultraparacompact.
$\eop_{\ref{boundedmetrisable}}$
\end{proof}

As mentioned above, Motto Ros proves in Theorem 5.6 of \cite{luca} that for $\kappa$ singular, the space ${}^{<<\kappa}\square\, 2^\kappa$
is not $(<\kappa)$-compact.

\section{Future directions: Combinatorics and logic in connection with paracompact cardinals}\label{sec:combo}
A natural approach to counter-attacking Conjecture \ref{nosingular} is to consider Prikry-type extensions, where $\kappa$ which is a large cardinal in the
ground model $V$ changes its cofinality and becomes singular. The large cardinal property in $V$ guarantees the compactness of the spaces of the form
${}^{<\kappa}\square\, 2^\lambda$, and one could hope that this will carry to at least paracompactness in the extension. This approach presents us with a
major problem, in that properties of open covers are not expressible in the first order logic and hence their preservation in the forcing extensions is
quite difficult to understand. Moreover, even the space ${}^{<\kappa}\square\, 2^\kappa$ will necessarily change in the extension, by the introduction of new subsets to $\kappa$. Hence, in order to obtain some handle on what could be done by forcing in the context of paracompact cardinals, it is desirable to represent the problem as a combinatorial property. It turns out indeed that in the case of spaces of the form 
${}^{<\kappa} \square\,2^\lambda$ b-ultraparacompactness is equivalent to a purely combinatorial property of the cardinals $\kappa$ and $\lambda$, as we show in Theorem \ref{combiparacompact} below. 

\begin{definition}\label{partial-funct} Let $\kappa$ be an infinite cardinal and $\lambda\ge\kappa$. 

{\noindent (1)} By ${\rm Fn}(\lambda,\kappa,2)$ we denote 
the set of all partial functions $s$ from $\lambda$ to 2, with $|\dom(s)|<\kappa$, ordered by $\subseteq$. Hence, an antichain in 
${\rm Fn}(\lambda,\kappa,2)$ consists of pairwise incompatible functions.

\smallskip
{\noindent (2)} A set $S\subseteq {\rm Fn}(\lambda,\kappa,2)$ is {\em dense} if
for every $f:\,\lambda \to 2$ there is $s\in S$ with $s\subseteq f$.

\smallskip
{\noindent (3)} For $R,S\subseteq {\rm Fn}(\lambda,\kappa,2)$ we say that {\em $R$ refines $S$} and write $S\le R$ iff for every $f\in R$ there is
$g\in S$ with $g\subseteq f$.

\end{definition}

\begin{theorem}\label{combiparacompact} Let $\kappa$ be an infinite cardinal and $\lambda\ge\kappa$. Then
${}^{<\kappa} \square\,2^\lambda$ is b-paracompact iff for every $S\subseteq {\rm Fn}(\lambda,\kappa,2)$ which is dense, there is a refinement $R\ge S$ which is an antichain in ${\rm Fn}(\lambda,\kappa,2)$ and still dense. \end{theorem}

\begin{proof} The proof is really just a matter of translation. In the forward direction, let $S\subseteq {\rm Fn}(\lambda,\kappa,2)$ be dense and for $s\in S$ let $[s]=\{f\in 2^\lambda:\, s\subseteq \lambda\}$ denote the
basic open set determined by $s$. Since $S$ is dense, the family $\mathcal S=\{ [s]:\, s\in S\}$ is a cover of ${}^{<\kappa} \square\,2^\lambda$  by basic open sets.
By b-paracompactness, there is a refinement $\mathcal R$ of $\mathcal S$ which consists of disjoint basic clopen sets. 

Let
$R=\{r\in {\rm Fn}(\lambda,\kappa,2):\, [r]\in \RR\}$. Since $\RR$ is a refinement of $\mathcal S$, we have that $S\le R$. Since $\RR$ is a cover we have that
$R$ is dense, and since the sets in $\RR$ are pairwise disjoint, the functions in $R$ are pairwise incompatible.

The proof of the other direction is similar.
$\eop_{\ref{combiparacompact}}$
\end{proof}

The combinatorial property from Theorem $\ref{combiparacompact}$ brings us one step closer to considering paracompact cardinals in forcing extensions,
but we have not been able to do much more for the moment. Perhaps that it can be approached in the future work. We turn to a final consideration, which is to do with logic. Namely, compactness of box products are closely connected with the compactness of abstract logics, in that for example, $\kappa$ is a strongly compact cardinal iff the logic $\LL_{\kappa,\kappa}$ is $(<\kappa)$-compact (this was in fact the original motivation for the definition of strongly compact cardinals, see \cite{Kanamori} for the full history and equivalences). A similar characterisation exists for weakly compact cardinals. Hence we can ask the
following question:

\begin{question}\label{paracopactinlogic} Suppose that $\kappa$ is a paracompact cardinal. What does this imply in terms of the logic 
$\LL_{\kappa,\kappa}$ ? 
\end{question}

\section{Appendix: Topological Reduced Products and the Nabla}\label{reduced_products}\label{ap:nabla}
The innovation of the proof of Kunen's Theorem (Theorem \ref{Kunenth}) is the use of the technique of reduced topological products, which was used again by Miller to prove Theorem \ref{T:M82}. We give generalisations of these proofs to uncountable cardinals, first taking a moment to review the technique.
The details are developed in the articles of Bankston \cite{MR458351} and \cite{MR0454898}.

\begin{definition}\label{def.topologicalreducedproduct} Let $\langle X_\alpha: \,\alpha
<\lambda\rangle$ be a sequence of topological spaces, and let $\D$ be a filter on $\lambda$. 
The \emph{topological reduced product via} $\D$ of the $X_\alpha$'s is 
\[
{}^{\D}\bigbox_{\alpha<\lambda} X_\alpha = \bigbox_{\alpha<\lambda} X_\alpha / \sim_{\D} \,,
\]
where for $x, y \in \bigbox_{\alpha<\lambda} X_\alpha$ we define 
\[x \sim_{\D} y \iff
\{\alpha: x_\alpha = y_\alpha\} \in \D.
\] 
\end{definition}

In Definition \ref{def.topologicalreducedproduct}, in the case $\D = \{\lambda\}$, we have
${}^\D\bigbox_{\alpha<\lambda} X_\alpha = \bigbox_{\alpha<\lambda} X_\alpha$. If
 $\D$ is an ultrafilter then ${}^\D\bigbox_{\alpha<\lambda} X_\alpha$ is referred to 
 as a \emph{topological ultraproduct}. 

 The natural projection $q_\D : \bigbox_{\alpha<\lambda} X_\alpha \to 
{}^\D\bigbox_{\alpha<\lambda} X_\alpha$ 
is clearly an open map, and 
\[
q_\D\left(\prod_{\alpha<\lambda} O_\alpha\right) = {}^\D\bigbox_{\alpha<\lambda} O_\alpha = 
\Big\{[x]_\D : \{\alpha: x_\alpha \in O_\alpha\}\in\D\Big\}
\] 
is called an \emph{open ultrabox}. One can note that if $[x]_\D \in  {}^\D\bigbox_{\alpha<\lambda} O_\alpha$ and 
$y \sim_{\D} x$, then $[y]_\D \in  {}^\D\bigbox_{\alpha<\lambda} O_\alpha$ also, since $\D$ has the finite intersection 
property. 

If $\B_\alpha$ is a basis for the 
topology on $X_\alpha$ and if $\D$ is any filter then 
${}^\D\bigbox_{\alpha<\lambda} \B_\alpha = \{{}^\D\bigbox_{\alpha<\lambda} O_\alpha: O_\alpha\in\B_\alpha 
\text{ for all }\alpha < \lambda\}$ is a basis for the reduced product topology.

\begin{example}\label{Ex:reduced}
Let $\D$ be the filter of \emph{cofinite subsets of} $\lambda$. In this case, $x \sim_{\D} y$ if and only if 
$x_\alpha = y_\alpha$ for all but a finite number of $\alpha$'s. For $\lambda = \aleph_0$ this is the 
reduced product in \cite{MR514975} or \cite{MR3414877}. 

Other examples include
the situation when $\kappa$ is any regular cardinal and $\D=\D_\kappa$ is the filter of sets whose complements have cardinality
$<\kappa$, which is the filter we shall work with.
\end{example}

From this point on, let us fix a regular cardinal $\kappa$, let $\D=\D_\kappa$ and suppose that $\lambda,\theta\ge\kappa$ (we shall mostly be interested in the cases $\lambda=\kappa$ or $\lambda=\kappa^+$ and similarly for $\theta$). Consider ${}^{<\theta} \bigbox_{\alpha<\lambda} X_\alpha$ for some spaces $\langle X_\alpha:\,\alpha <\lambda\rangle$. 

\begin{definition}[Nabla]\label{nabla} Let $\sim_\kappa$ be the equivalence relation on $\prod_{\alpha<\lambda} X_\alpha$
defined by 
\[
x\sim_\kappa y\iff |\{\alpha: x_\alpha\neq y_\alpha\}| < \kappa.
\]
The space ${}^{<\theta} \nabla_{\alpha<\lambda} X_\alpha =  {}^{<\theta} \bigbox_{\alpha<\lambda} X_\alpha/{\sim_\kappa}$ is the quotient topology on $X(\theta,\lambda)$ induced by the equivalence classes 
$\{[ x]_{\sim_\kappa} :\, x\in {}^{<\theta} \bigbox_{\alpha<\lambda} X_\alpha\}$. Let $q:\,{}^{<\theta} \bigbox_{\alpha<\lambda} X_\alpha
\to {}^{<\theta} \nabla_{\alpha<\lambda} X_\alpha$ be the quotient map.

If $Y\subseteq {}^{<\theta} \bigbox_{\alpha<\lambda} X_\alpha$ we define $[Y] = \{[ y]_{\sim_\kappa}:\, y\in Y\}$. A basis of the space ${}^{<\theta} \nabla_{\alpha<\lambda} X_\alpha$
consists of all $[ U]$, where $U ={}^{<\theta} \bigbox_{\alpha<\lambda} X_\alpha$ is an open box in  ${}^{<\theta} \bigbox_{\alpha<\lambda} X_\alpha$.
\end{definition}

\subsection{The case $\lambda=\kappa$, $\theta=\kappa^+$.}\label{subsec:tau} Let $X=\bigbox_{\alpha<\kappa} X_\alpha$.
This case with $\kappa=\aleph_0$ corresponds to the forward direction of Kunen's Theorem (Theorem \ref{Kunenth}).

For each $\alpha<\kappa$, let $\tau_\alpha$ be the canonical 
projection from $X = \bigbox_{\beta < \alpha}X_\beta \times 
\bigbox_{\alpha \leq \beta < \lambda}X_\beta$ onto $\bigbox_{\alpha\leq \beta < \kappa}X_\beta$.

\begin{lemma}\label{L:6.10d} Suppose that for every $\alpha<\kappa$, the space $X_\alpha$ is $(<\kappa)$-open and
$U_\alpha$ is an open subset of $\bigbox_{\alpha\leq \beta < \kappa}X_\beta$.
Then $\bigcap_{\alpha <\kappa} \tau_\alpha^{-1}(U_\alpha)$ is open in $X$.
\end{lemma}

\begin{proof} If $\bigcap_{\alpha <\kappa} \tau_\alpha^{-1}(U_\alpha)=\emptyset$, then we immediately have the desired conclusion, so let us assume that
$\bigcap_{\alpha <\kappa} \tau_\alpha^{-1}(U_\alpha)\neq \emptyset$. Let $p\in \bigcap_{\alpha <\kappa} \tau_\alpha^{-1}(U_\alpha)$. For every $\alpha < \kappa$, the projection $\tau_\alpha$ is continuous, so there exists a basic open neighbourhood $V_\alpha$ of $p$ contained in $\tau_\alpha^{-1}(U_\alpha)$.
The neighbourhood $V_\alpha$ can be taken to be of the form 
\[
V_\alpha = \bigbox_{\beta < \alpha}X_\beta \times 
\bigbox_{\alpha \leq \beta < \kappa}V^\alpha_\beta \,,
\]
where each $V^\alpha_\beta$ is open in $X_\beta$ for $\alpha \leq \beta < \kappa$. For $\beta<\kappa$ let
$W_\beta = \bigcap_{\alpha\leq \beta} V^\alpha_\beta$, which is open in $X_\beta$ since $X_\beta$ is 
$(<\kappa)$-open. Let $W = \bigbox_{\beta < \kappa}W_\beta$, which is open in $X$. However,
$W = \bigcap_{\alpha < \kappa} V_\alpha$, so $W$ is an open neighbourhood of $p$ contained in each
$\tau_\alpha^{-1}(U_\alpha)$.
$\eop_{\ref{L:6.10d}}$
\end{proof}

\begin{corollary}\label{C:6.11d} Suppose that for each $\alpha<\kappa$, the space $X_\alpha$ is $(<\kappa)$-open. Then 
\begin{enumerate}[{\rm (a)}]
\item $\nabla_{\alpha<\kappa} X_\alpha$ is $(<\kappa^+)$-open,
\item If $\chi(X_\alpha)\leq \kappa$ for all $\alpha<\kappa$, then $\nabla_{\alpha<\kappa} 
X_\alpha$ is $(<\mathfrak{b}({\kappa}))$-open.
\end{enumerate}
\end{corollary}

\begin{proof} (a)
Let $q_\alpha: \bigbox_{\alpha\leq\beta <\kappa} 
X_\beta \to \nabla_{\alpha<\kappa} X_\alpha$ be the unique map satisfying 
$q = q_\alpha \circ \tau_\alpha$. 

Suppose that we are given sets 
$V_\alpha\,(\alpha<\kappa)$ each of which is open in $\nabla_{\alpha<\kappa} X_\alpha$. Then $q_\alpha^{-1}(V_\alpha)$ is open in $\bigbox_{\alpha\leq\beta < \kappa}X_\beta$, 
so $q^{-1}\big(\bigcap_{\alpha <\kappa} V_\alpha\big)$, which is equal to
$\bigcap_{\alpha <\kappa} \tau_\alpha^{-1}\big(q_\alpha^{-1}(V_\alpha)\big)$ is open in $X$ by Lemma \ref{L:6.10d}.
Consequently, $\bigcap_{\alpha < \kappa} V_\alpha$ is 
open in $\nabla_{\alpha<\kappa} X_\alpha$.

\medskip

{\noindent (b)} Let $p\in X$ and let $\{U_\beta:\beta <\mu\}$ be open neighbourhoods of $q(p)$ in 
$\nabla_{\alpha<\kappa} X_\alpha$, where $\mu <\mathfrak{b}({\kappa})$. 
We show that $\bigcap_{\beta < \mu}U_\beta$ is an open neighbourhood of $q(p)$.

By the assumption that $\chi(X_\alpha)\leq \kappa$ for all $\alpha<\kappa$, for each $\alpha < \kappa$ we can find
a point base of size $\kappa$ at the 
point $p_\alpha$ in $X_\alpha$. We enumerate such a base as $\langle V_\gamma^\alpha: \gamma < \kappa\rangle$ Since each $X_\alpha$ is $(<\kappa)$-open, we can assume 
that $V^\alpha_\gamma \subseteq V^\alpha_{\gamma'}$ whenever $\gamma' \le \gamma$. 
For each $\beta < \mu$, let $f_\beta:\kappa \to \kappa$ be such that
\[
q\left( \bigbox_{\alpha<\kappa} V^\alpha_{f_\beta(\alpha)} \right) \subseteq U_\beta, 
\]
which is well-defined by the definition of the topology.

Since $\mu< \mathfrak{b}(\kappa)$, there exists a function $g:\,\kappa\to\kappa$ with $f_\beta^\ast\le_\kappa g$ for all
$\beta<\mu$. For each such $\beta$, let $\gamma_\beta<\kappa$ be such that for all $\gamma\ge\gamma_\beta$,
$f_\beta(\gamma)\le g(\gamma)$. We then have that for each $\beta<\kappa$,
\[
\bigbox_{\alpha<\kappa} V^\alpha_{g(\alpha)}=\bigbox_{\alpha<\gamma_\beta} V^\alpha_{g(\alpha)}\times \bigbox_{\gamma_\beta\le\alpha<\kappa} V^\alpha_{g(\alpha)}\subseteq \bigbox_{\alpha<\gamma_\beta} V^\alpha_{g(\alpha)}\times \bigbox_{\gamma_\beta\le\alpha<\kappa} V^\alpha_{f_\beta(\alpha)}.
\]
Therefore $q\left( \bigbox_{\alpha<\kappa} V^\alpha_{g(\alpha)} \right)\subseteq q\left( \bigbox_{\alpha<\kappa} V^\alpha_{f_\beta(\alpha)} \right)$
because every point in $\bigbox_{\alpha<\kappa} V^\alpha_{g(\alpha)}$ is $\sim_\kappa$-equivalent to a point in 
$\bigbox_{\alpha<\gamma_\beta} V^\alpha_{f_\beta(\alpha)}\times \bigbox_{\gamma_\beta\le\alpha<\kappa} V^\alpha_{f_\beta(\alpha)}=\bigbox_{\alpha<\kappa} V^\alpha_{f_\beta(\alpha)}$.
Hence, $q\left( \bigbox_{\alpha<\kappa} 
V^\alpha_{g(\alpha)} \right)$ is an open neighbourhood of $q(p)$ contained in 
$\bigcap_{\beta < \mu}U_\beta$.
$\eop_{\ref{C:6.11d}}$
\end{proof}

\subsection{The case $\theta=\lambda=\kappa^+$.} Let $X={}^{<\kappa^+} \bigbox_{\alpha<\kappa^+} X_\alpha$.
This case with $\kappa=\aleph_0$ corresponds to the compact metric case of Miller's Theorem (Theorem \ref{T:M82}). 

\begin{lemma}\label{lm:openess} Suppose that for each $\alpha<\kappa^+$, the space $X_\alpha$ is $(<\kappa)$-open. Then 
${}^{<\kappa^+}\nabla_{\alpha<\kappa^+} X_\alpha$ is $(<\kappa^+)$-open.
\end{lemma}

\begin{proof} Suppose that $\{V_\beta:\,\beta<\kappa\}$ are open in ${}^{<\kappa^+}\nabla_{\alpha<\kappa^+} X_\alpha$. If $\bigcap_{\beta<\kappa} V_\beta=\empty$, then the intersection is an
open set, so suppose that $[x] \in \bigcap_{\beta<\kappa} V_\beta$. For each $\beta < \kappa$ 
let $U^\beta =  {}^{<\kappa^+}\!\bigbox_{\alpha<\kappa^+}  U^\beta_\alpha$ a neighbourhoods of $x$ in  $X$
satisfying $q(U^\beta) \subseteq V_\beta$. Therefore $|\bigcup_{\beta<\kappa}\supt(U_\beta)|\le\kappa$ and it can be enumerated, possibly with repetitions
as $\{\alpha_\gamma:\,\gamma<\kappa\}$.

For each $\gamma<\kappa$ the set
$W_{\alpha_\gamma} = \bigcap_{\beta\leq \gamma} U^\beta_{\alpha_\gamma}$ is open in 
$X_{\alpha_\gamma}$, since each $X_\alpha$ is $(<\kappa)$-open. Set $W_\alpha = X_\alpha$ for all 
$\alpha \notin \{\alpha_\gamma:\gamma<\kappa\}$. Then $x\in W = \prod_{\alpha < \kappa^+} W_\alpha$ 
and $q(W)\subseteq q(U^\beta)$ for all $\beta < \kappa$. Consequently, $q(W) \subseteq 
\bigcap_{\beta<\kappa}V_\beta$.
$\eop_{\ref{lm:openess}}$
\end{proof}

\bibliographystyle{plain}
\bibliography{../../../bibliomaster}

\end{document}